\documentclass[EJP]{ejpecp} 

\usepackage{tikz}
\usetikzlibrary{patterns}  
 \usepackage{enumerate}  
\usepackage{graphicx}
\usepackage[font=small,labelfont=bf]{caption}
\usepackage{epstopdf}
\usepackage{xcolor} 
\usepackage{float} 
\usepackage{listings}
\usepackage{longtable}
\usepackage{mathrsfs}
\usepackage{dsfont}

\def\fddto{\xrightarrow{\textit{f.d.d.}}}
\newcommand{\ind}{{\bf 1}}

\def\inddd#1{{\ind}_{\left\{#1\right\}}} 
\newcommand{\proba}{\mathbb P}
\newcommand{\esp}{{\mathbb E}}

\newcommand{\argmin}{{\rm{argmin}}}
 
\newcommand{\eqnh}{\begin{eqnarray*}}
\newcommand{\eqne}{\end{eqnarray*}}
\newcommand{\eqnhn}{\begin{eqnarray}}
\newcommand{\eqnen}{\end{eqnarray}}
\newcommand{\equh}{\begin{equation}}
\newcommand{\eque}{\end{equation}} 
\def \ds {\displaystyle} 
\def\summ#1#2#3{\sum_{#1 = #2}^{#3}}
\def\prodd#1#2#3{\prod_{#1 = #2}^{#3}}

\newcommand{\eqd}{\stackrel{d}{=}}
 
\def\topp#1{^{(#1)}}

\def\abs#1{\left|#1\right|}
\def\sabs#1{|#1|} 
\def\ccbb#1{\left\{#1\right\}} 

\def\pp#1{\left(#1\right)}
\def\spp#1{(#1)}
\def\bb#1{\left[#1\right]}

\def\mmid{\;\middle\vert\;}

\def\floor#1{\left\lfloor #1 \right\rfloor}

\def\aa#1{\left\langle #1\right\rangle}

\def\vv#1{{\boldsymbol #1}}

\def\qmand{\quad\mbox{ and }\quad}

\def\qmwith{\quad\mbox{ with }\quad}
\def\mfa{\mbox{ for all }}

\def\wt#1{\widetilde{#1}}

\def\what#1{\widehat{#1}}
 
\def\limn{\lim_{n\to\infty}}

\def\weakto{\Rightarrow} 
\def\Z{{\mathbb Z}}

\def\R{{\mathbb R}}
\def\Rd{{\mathbb R^d}} 
\def\N{{\mathbb N}} 
\def\BB{{\mathbb B}}

\def \lb {\left(}
\def \rb {\right)} 
\def\qp#1{\lb #1;q\rb_\infty}

\def\qps#1{\lb #1;q\rb}
\def\qpp#1{\left(#1;q\right)_\infty}
\renewcommand{\i}{{\rm i}}
\renewcommand{\d}{{\rm d}}

\newcommand{\calH}{{\mathcal H}}

\newcommand{\calZ}{\mathcal{Z}} 
 
\def\<{\langle}
\def\>{\rangle} 
\def\p{\mathsf p} 
\def\q{{\mathsf q}}
\def\A{{\mathsf a}}

\def\C{{\mathsf c}}

\newcommand{\PhiB}{\Phi^{\BB}} 
\newcommand{\vvx}{{\vv x}}
\newcommand{\vvc}{{\vv c}}
\newcommand{\vvu}{{\vv u}}

\newcommand{\vvt}{{\vv t}}
\newcommand{\vvz}{{\vv z}} 
\def \lee {\left[}
\def \ree {\right]} 

\newcommand{\erfc}{{\rm erfc}}
\newcommand{\rmc}{{\rm c}}
\newcommand{\rmd}{{\rm d}}

\SHORTTITLE{From open ASEP to open KPZ fixed point}

\TITLE{From asymmetric simple exclusion processes with open boundaries to stationary measures of open KPZ fixed point:  the  shock region\support{Y.W.~was partially supported by Army Research Office, US (W911NF-20-1-0139) and Simons Foundation (MP-TSM-00002359). Z.Y.~was partially supported by Ivan Corwin’s NSF grant DMS-1811143 as well as the Fernholz Foundation’s `Summer Minerva Fellows' program.} }

\AUTHORS{ 
  Yizao~Wang\footnote{ University of Cincinnati, United States of America. \EMAIL{yizao.wang@uc.edu}} 
  \and 
  Zongrui~Yang\footnote{Columbia University, United States of America. \EMAIL{zy2417@columbia.edu}  }}


\AMSSUBJ{60F05; 60K35}  

\SUBMITTED{ }  
\ACCEPTED{ }  

\VOLUME{ }
 
\YEAR{2025}
\PAPERNUM{ }
\DOI{ }

\ABSTRACT{
We continue the investigation of limit fluctuations of stationary measures of the asymmetric simple exclusion processes with open boundaries  (open ASEP), complementing the recent result by \citet{bryc23asymmetric}. It was shown therein that in the fan region of the phase diagram, an appropriate scaling limit of the height function of open ASEP converges in distribution to a stochastic process introduced by \citet{barraquand22steady}, known as the stationary measure of the (conjectural) open KPZ fixed point. In this paper, we establish the corresponding  convergence in the shock region. Our proof is based on the integral representation of the matrix product ansatz in terms of Askey–Wilson signed measures introduced by \citet{wang24askey}. The analysis of the asymptotic behavior is more delicate here than in the fan region. In particular, our proof of the duality formula for the stationary measures of the open KPZ fixed point in the shock region is different from the approach by~\citep{bryc23asymmetric} taken in the fan region: ours relies crucially on a recent result by~\citet{bryc24two}.
}

\begin{document} 
\sloppy

\section{Introduction and main result}  

\subsection{Preface} 
Asymmetric simple exclusion processes with open boundaries (open ASEP) serve as  paradigmatic models 
for nonequilibrium systems with open boundaries and for Kardar--Parisi--Zhang (KPZ) universality. Over the past 50 years, extensive studies have been dedicated to understanding its stationary measures,
 encompassing a wide range of asymptotic 
  behaviors 
  concerning  particle densities, 
  fluctuations, correlation functions, and large deviations. See \citep{blythe07nonequilibrium,derrida06matrix,derrida07nonequilibrium,corwin22some} and more references therein for early developments. A significant portion of these studies is based on the powerful matrix product ansatz approach, first introduced by Derrida, Evans, Hakim, and Pasquier in 1993 \citep{derrida93exact}. This method is notably related to the Askey--Wilson polynomials \citep{uchiyama04asymmetric,corteel11tableaux} and processes \citep{bryc17asymmetric,wang24askey}. It is widely known that the asymptotic behaviors of open ASEP exhibit a phase diagram involving three phases: maximal current, low density, and high density. The phase diagram can also be divided into the fan region and the shock region, based on the behavior of the system after running for an intermediate length of time.

In this paper, we continue recent studies on the second-order limit fluctuations of stationary measures of open ASEP. These fluctuations are believed to exhibit universal phenomena in the KPZ class for nonequilibrium systems with open boundaries.  
The {\em open KPZ fixed point} is currently a conjectural object that plays a role analogous to the KPZ fixed point 
(see \citet{matetski21KPZ}) when restricted to an interval. 
The work by \citet{barraquand22steady} postulated the stationary measures of the (conjectural) open KPZ fixed point, which are expected to arise as the scaling limits of stationary measures of all models in the KPZ class on an interval. 
These postulated stationary measures of open KPZ fixed point 
 depend on two boundary parameters, $\A$ and $\C$ (both in $\R$), and can be defined as the sum of two independent processes 
 indexed by $t\in[0,1]$:
 $\BB+\eta\topp{\A,\C}$ where
 $\BB$ is a Brownian motion and $\eta\topp{\A,\C}$ is a stochastic process. The process $\eta\topp{\A,\C}$ has many representations: 
its law can be characterized as a change of measure of the law of a standard Brownian motion with respect to an explicit Radon--Nikodym derivative, or via explicit joint density formulas. 
In the cases when 
 $\A+\C\ge 0$ and $A_nC_n<1$ (corresponding to the fan region of open ASEP)
a recent work by \citet{bryc23asymmetric} established the convergence of the fluctuations of 
stationary measures of open ASEP
 to 
$\BB+\eta\topp{\A,\C}$.

The contribution of this paper is the extension of the limit theorem in \citep{bryc23asymmetric} from the fan region to the shock region, and hence provide a complete picture regarding the limit fluctuations for height function of stationary measures in the entire phase diagram.

\subsection{Asymmetric simple exclusion process with open boundaries}
An open ASEP
of size $n\in\N = \{1,2,\dots\}$
 is a continuous-time irreducible Markov process on state space $\{0,1\}^n$ with five parameters
\equh\label{eq:conditions open ASEP}
\alpha,\beta>0,\quad \gamma,\delta\geq 0,\quad 0\leq q<1,
\eque
which models the evolution of particles on the lattice $\left\{1,\dots,n\right\}$. Particles move 
to the left nearest neighbor with rate $q$ and to the  right nearest neighbor with rate $1$, if the target site is empty. Moreover, particles enter the system  from left at site $1$  with rate $\alpha$ and from right at site $n$ with rate $\delta$, provided that the site is empty. A particle at site $1$  leaves the system 
  with rate $\gamma$, and a particle at site $n$ leaves the system 
   with rate $\beta$. These jump rates are summarized in Figure \ref{fig:openASEP}.  When $q=\gamma=\delta=0$, particles can only move to the right, and the system is referred to as the totally asymmetric simple exclusion process with open boundaries (open TASEP).
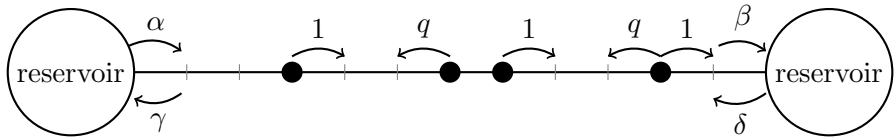
\begin{figure}[ht]
\centering
\begin{tikzpicture}[scale=0.7]
\draw[thick] (-1.2, 0) circle(1.2);
\draw (-1.2,0) node{reservoir};
\draw[thick] (0, 0) -- (12, 0);
\foreach \x in {1, ..., 12} {
	\draw[gray] (\x, 0.15) -- (\x, -0.15);
}
\draw[thick] (13.2,0) circle(1.2);
\draw(13.2,0) node{reservoir};
\fill[thick] (3, 0) circle(0.2);
\fill[thick] (6, 0) circle(0.2);
\fill[thick] (7, 0) circle(0.2);
\fill[thick] (10, 0) circle(0.2);
\draw[thick, ->] (3, 0.3)  to[bend left] node[midway, above]{$1$} (4, 0.3);
\draw[thick, ->] (6, 0.3)  to[bend right] node[midway, above]{$q$} (5, 0.3);
\draw[thick, ->] (7, 0.3) to[bend left] node[midway, above]{$1$} (8, 0.3);
\draw[thick, ->] (10, 0.3) to[bend left] node[midway, above]{$1$} (11, 0.3);
\draw[thick, ->] (10, 0.3) to[bend right] node[midway, above]{$q$} (9, 0.3);
\draw[thick, ->] (-0.1, 0.5) to[bend left] node[midway, above]{$\alpha$} (0.9, 0.4);
\draw[thick, <-] (0, -0.5) to[bend right] node[midway, below]{$\gamma$} (0.9, -0.4);
\draw[thick, ->] (12, -0.4) to[bend left] node[midway, below]{$\delta$} (11, -0.5);
\draw[thick, <-] (12, 0.4) to[bend right] node[midway, above]{$\beta$} (11.1, 0.5);
\end{tikzpicture}
\caption{Jump rates in the open ASEP.}
\label{fig:openASEP}
\end{figure}

Throughout the paper we assume condition \eqref{eq:conditions open ASEP} and work with the following parameterization. With
 \[
\kappa_{\pm}(x,y):=\frac{1}{2x}\pp{1-q-x+y\pm\sqrt{(1-q-x+y)^2+4xy}}, \quad \mbox{for }\; x>0\mbox{ and }y\geq 0,
\] 
we set
\equh\label{eq:defining ABCD}
A=\kappa_+(\beta,\delta),\quad B=\kappa_-(\beta,\delta),\quad C=\kappa_+(\alpha,\gamma),\quad D=\kappa_-(\alpha,\gamma).
\eque
One can check that \eqref{eq:defining ABCD} gives a bijection between \eqref{eq:conditions open ASEP} and 
\equh\label{eq:conditions qABCD}
A,C\geq0,\quad -1<B,D\leq 0,\quad 0\leq q<1.
\eque
We will assume \eqref{eq:conditions open ASEP} and consequently \eqref{eq:conditions qABCD} throughout the paper. 
The quantities 
$A/(1+A)$ and 
$1/(1+C)$
 have nice
physical interpretations as the `effective densities' near the left and right boundaries of the system.

It has been known since \citep{derrida93exact} that the phase diagram of open ASEP involves only two boundary parameters $A,C$, and exhibits three phases:
\begin{itemize}
        \item  (maximal current phase) $A<1$, $C<1$,
        \item  (high density phase) $A>1$, $A>C$,
        \item  (low density phase) $C>1$, $C>A$.
    \end{itemize}
The boundary $A=C>1$ between the high and low density phases is called the coexistence line.

There are also two regions on the phase diagram distinguished by \citep{derrida02exact,derrida03exact}:
    \begin{itemize}
        \item (fan region) $AC<1$,
        \item  (shock region) $AC>1$.
    \end{itemize}
See Figure \ref{Fig3} for an illustration.

 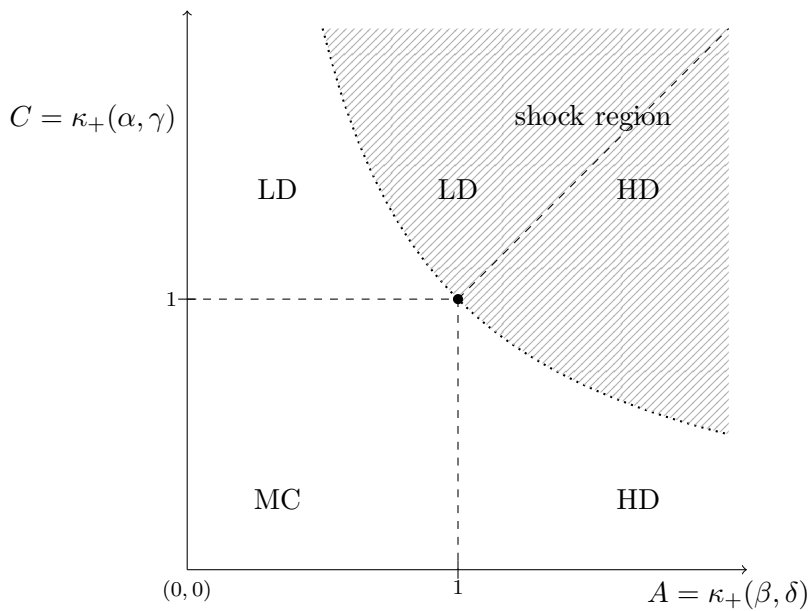
\begin{figure}[tb]
\centering
  \begin{tikzpicture}[scale=1]

\draw[scale = 1,domain=6.5:11,smooth,variable=\x,dotted,thick] plot ({\x},{1/((\x-7)*1/3+2/3)*3+5});

\fill[pattern=north east lines, pattern color=gray!60] (11,11)--(6.5,11) 
-- plot [domain=6.5:11]  ({\x},{1/((\x-7)*1/3+2/3)*3+5});

 \draw[->] (5,5) to (5,11.2);
 \draw[->] (5.,5) to (11.2,5);
   \draw[-, dashed] (5,8) to (8,8);
   \draw[-, dashed] (8,8) to (8,5);
   \draw[-, dashed] (8,8) to (11,11);
   \node [left] at (5,8) {\scriptsize$1$};
   \node[below] at (8,5) {\scriptsize $1$};
     \node [below] at (11,5) {$A = \kappa_+(\beta,\delta)$};
   \node [left] at (5,10) {$C = \kappa_+(\alpha,\gamma)$};

\node[above] at (9.5,9.75) {shock region};

  \draw[-] (8,4.9) to (8,5.1);
   \draw[-] (4.9,8) to (5.1,8);

 \node [below] at (5,5) {\scriptsize$(0,0)$};
    \node [above] at (6,9) {LD};
    \node[above] at (8,9) {LD};
    \node [below] at (10,6) {HD};  
      \node [above] at (10,9) {HD};

 \node [below] at (6,6) {MC}; 

\draw [fill=black] (8,8) circle [radius=0.05];
\end{tikzpicture}
\caption{
Phase        diagram for the open ASEP with maximal current (MC), low density (LD), high density (HD) phases, and with shaded  shock region $AC>1$.
Under Assumption \ref{assump:0} 
and $\A+\C<0$, parameters converge from within the shock region to the triple point (1,1) as $n\to\infty$.
}\label{Fig3}
\end{figure}

\subsection{Main Result}\label{subsec:main results}
Now we explain the setup of our limit theorem. We consider a sequence of open ASEP, each with size $n\in\N$, 
fixed $q\in[0,1)$ and changing parameters $A_n, B_n, C_n,D_n$ (that is, the parameters $\alpha,\beta,\gamma,\delta$ 
now depend on $n$, and we use the transformation \eqref{eq:defining ABCD}). 
We will be interested  the convergence to the ``triple point” in the phase diagram, 
that is, when 
$(A_n,C_n)\to (1,1)$, the point
at the intersection of three phases for the open ASEP. For parameters $B_n$ and $D_n$ we will assume that they converge respectively to constants $B$ and $D$ in $(-1,0]$.
We shall impose the following assumptions.
\begin{assumption}\label{assump:0}
We assume $A_n,C_n\ge 0$, $B_n,D_n\in(-1,0]$ for all $n\geq 1$,  
\begin{align*}
\limn A_n = 1 & \mbox{ with }  \limn \sqrt n(1-A_n) = \A\in
(-\infty,\infty),  
\\
\limn C_n = 1& \mbox{ with }  \limn \sqrt n(1-C_n) = \C\in (-\infty,\infty).  
\end{align*}
We also assume that
\[ 
    \limn B_n=B\in(-1,0]  \qmand   \limn D_n=D\in(-1,0].
\] 
\end{assumption}

\begin{remark}
Under Assumption \ref{assump:0}, we observe that when $\A+\C > 0$, we have $A_nC_n < 1$ for sufficiently large $n$, i.e., the open ASEP is in the fan region. Similarly, when $\A+\C < 0$, we have $A_nC_n > 1$ for sufficiently large $n$, i.e., the open ASEP is in the shock region. When $\A+\C = 0$, $A_nC_n$ may fluctuate around $1$ for sufficiently large $n$, meaning the open ASEP can alternate between the fan and shock regions.
\end{remark}

For each $n$ fixed, consider open ASEP with parameters $A_n,B_n,C_n,D_n,q$ as above. There exists a unique stationary measure, denoted by $\mu_n$, of the open ASEP on the configuration space $\{0,1\}^n$, and we let $(\tau_{n,1},\dots,\tau_{n,n})\in\{0,1\}^n$ denote a random vector with law $\mu_n$. Consider the (centered) height function:
\equh\label{eq:h_n}
h_n(x):=\summ k1{\floor{nx}}(2\tau_{n,k}-1), \quad x\in[0,1].
\eque

Next, we introduce the stochastic processes that arise in the limit. 
For all $\A,\C\in\R$, we let $\eta\topp{\A,\C}= (\eta\topp{\A,\C}_t)_{t\in[0,1]}$
denote a stochastic process 
such that 
 its law on $C([0,1])$, denoted by $\mathbb{P}_{\eta\topp{\A,\C}}$, is determined by the following Radon--Nikodym representation:
\equh\label{eq:RN1}
\frac{\d \mathbb{P}_{\eta\topp{\A,\C}}}{\d \mathbb{P}_{\rm Bm}}(\vv\omega) = \frac{\ds e^{(\A+\C)\min_{t\in[0,1]}\omega_t - \A\omega_1}}{\ds \esp_{\rm Bm}\lee e^{(\A+\C)\min_{t\in[0,1]}\omega_t - \A\omega_1}\ree},\quad \vv\omega= \lb\omega_t \rb_{t\in[0,1]}\in C([0,1]),
\eque
where $\mathbb{P}_{\rm Bm}$ denotes the law on $C([0,1])$ of a standard Brownian motion, and $\mathbb{E}_{\rm Bm}$ denotes the expectation with respect to $\mathbb{P}_{\rm Bm}$.
The normalization constant above
can be evaluated explicitly (see Appendix \ref{app:normalization}):
\equh\label{eq:calH}
\mathbb{E}_{\rm Bm}\lee e^{(\A+\C)\min_{t\in[0,1]}\omega_t - \A\omega_1}\ree=\calH(\A,\C):= \begin{cases}
\displaystyle \frac{\A H(\A/\sqrt 2)-\C H(\C/\sqrt 2)}{\A-\C}, & \mbox{ if } \A\ne \C,\\\\
\displaystyle (1+\A^2)H\pp{\frac {\A}{\sqrt 2}}-\sqrt{\frac2\pi}\A, & \mbox{ if } \A=\C,
\end{cases}
\eque
where we denote: 
\[ 
H(x) =e^{x^2}\erfc(x) \qmwith \erfc (x) = \frac 2{\sqrt\pi}\int_x^\infty e^{-t^2}\d t,\quad x\in\R.
\]

The following theorem provides a complete picture of limit fluctuations under Assumption~\ref{assump:0}. 
Our contribution here is the proof for the case $\A+\C<0$, corresponding to the shock region.
\begin{theorem}\label{thm:1}
Under Assumption \ref{assump:0},
with $\A,\C\in\R$,
we have 
\equh\label{eq:fdd}
\frac1{\sqrt n}\lb h_n(x)\rb_{x\in[0,1]}\weakto \frac1{\sqrt 2}\lb\BB_x + \eta\topp{\A/\sqrt 2,\C/\sqrt 2}_x\rb_{x\in[0,1]}\quad\mbox{ as }n\to\infty
\eque
as processes in $D([0,1])$, the space of c\`adl\`ag functions with Skorokhod topology \citep{billingsley99convergence}.
Here $\BB$ is a standard Brownian motion independent from the process $\eta\topp{\A/\sqrt 2,\C/\sqrt 2}$.
\end{theorem}
In the fan region $\A+\C>0$ and also when $\A+\C=0$, 
\eqref{eq:fdd} was established in \citep{bryc23asymmetric}. 
 Here we use a slightly different parameterization, and our $\eta\topp{\A/\sqrt 2,\C/\sqrt 2}$ here is $\eta\topp{\A,\C}$ therein. In fact, $\A = \infty$ and/or $\C = \infty$ was also investigated therein (with a limit theorem established for convergence of finite-dimensional distributions only); see Remark \ref{rem:ABCD} for a discussion of these cases. 
Note also that in \citep{bryc23asymmetric}, when $\A+\C = 0$ the proof is restricted to the fan region ($A_nC_n<1$) although the same argument also works for the shock region ($A_nC_n>1$). 
For the entire phase diagram $\A,\C\in\R$ but restricted to open TASEP ($q=\gamma_n=\delta_n=0$), Theorem \ref{thm:1} was proved by \citet{bryc24two}; see Remark \ref{rmk:limit process} for more details. Our method is in the same spirit as the one in \citep{bryc23asymmetric}, but completely different from the one in \citep{bryc24two}.

The main technique in the proof is an integral representation of the joint generating function of the stationary measure $\mu_n$ of open ASEP; see Section \ref{sec:AWP}. 
This representation, introduced by \citet{bryc17asymmetric} (based on earlier works 
 \citep{derrida93exact, uchiyama04asymmetric, bryc10askey}), was originally restricted to the fan region only. It is worth noting that in this region, the integral representation has a probabilistic interpretation as the expectation of a functional of the so-called Askey–Wilson Markov process, introduced by \citet{bryc10askey}. Recently, 
this representation was extended in \citep{wang24askey}
 to the shock region, involving the so-called Askey–Wilson signed measures. Although the probabilistic interpretation is no longer valid in the shock region, the asymptotic analysis remains similar. 
 
 In particular, this integral representation yields a 
  formula for the Laplace transform of the finite-dimensional distribution of the height function:

\begin{equation}\label{eq:Laplace transform intro} 
\esp\lee\exp\pp{-\summ k1d \frac{c_k}{\sqrt{n}}  h_n(x_k) }\ree,\quad c_1,\dots,c_d\in\R,
\end{equation}
and this formula is convenient for asymptotic analysis (see Theorem \ref{thm:AW rep}). 

The majority of the proof is devoted to calculating the limit Laplace transform described above, which is provided in Section \ref{sec:proof}. The analysis becomes more involved than in the fan region \citep{bryc23asymmetric}, and consequently the limit of equation \eqref{eq:Laplace transform intro} is much more complicated.  
In particular, the limit expression of \eqref{eq:Laplace transform intro} cannot be read immediately as the Laplace transform of the corresponding finite-dimensional distribution of the limit process on the right-hand side of \eqref{eq:fdd}. 
This step, summarized in Proposition \ref{prop:duality},
 is referred to as the establishment of the so-called {\em duality formula} for the Laplace transforms of Markov processes (see, for example, \citep{bryc23markov, bryc23asymmetric} and Remark \ref{rem:duality}).
 Unexpectedly, unlike in several other cases we have encountered in the past, we are unable to prove 
the duality formula in this case 
 by a direct calculation,
  except when $d=1$ 
  (see Proposition \ref{prop:d=1}).
   Instead, we provide a soft argument 
   for Proposition \ref{prop:duality} that relies crucially on the recent development by \citet{bryc24two}. 
Further comments regarding the duality formula can be found in Appendix \ref{sec:duality}. 

At last, for tightness it suffices to adapt a coupling argument in \citep[Proposition 2.4]{bryc23asymmetric}. This coupling argument is quite general and is independent from the approach by Askey--Wilson signed measures.

We conclude the introduction with a few remarks.

\begin{remark}\label{rmk:limit process}
The
process $\eta\topp{\A,\C}$ was introduced by \citet{barraquand22steady}. It was predicted therein that the right-hand side of \eqref{eq:fdd} is the stationary measure of the (conjectural) open KPZ fixed point, 
 and that it arises as the scaling limit of stationary measures of all models in the KPZ universality class with two open boundaries. We will present two alternative representations of the process $\eta^{(\A,\C)}$ in Section \ref{sec:eta}. In particular, Proposition \ref{prop:two line} demonstrates that for the two processes:
\[
\pp{ \frac1{\sqrt 2}\pp{\eta_x\topp{\A/\sqrt 2,\C/\sqrt 2}+\BB_x}_{x\in[0,1]}, \frac1{\sqrt 2}\pp{\eta_x\topp{\A/\sqrt2,\C/\sqrt2}-\BB_x}_{x\in[0,1]}},
\]
their joint law, a probability measure on $C([0,1])^2$, also has a representation in terms of a Radon--Nikodym derivative with respect to the law of two independent Brownian motions (see \eqref{eq:RN2}). This representation is often referred to as a two-line/two-layer
 representation. 
The discrete counterpart of the two-line representation was recently established for open TASEP by \citet{bryc24two}, and it led, via a much simpler proof, to Theorem \ref{thm:1} restricted to open TASEP ($\A,\C\in\R$, $q=\gamma_n=\delta_n = 0$). 
   At this moment, it is not clear to us whether the two-line representation of open TASEP can be extended to open ASEP with full range of parameters, while the methods in \citep{bryc23asymmetric} and the present paper (based on the Askey--Wilson polynomials) allow us to prove limit theorems for all legitimate parameters.  
   The representation in \citep{bryc24two} was actually inspired by another very recent result on some integrable polymer models on a diagonal strip by \citet{barraquand24stationary}, where a two-line representation was introduced and played a key role in showing the convergence of  stationary measures of these models to the stationary measures of open KPZ fixed point and of open KPZ equation. 
\end{remark}

\begin{remark}
Limit fluctuations for stationary measures of open ASEP were first studied with all parameters $A,B,C,D,q$ fixed and with system size $n\to\infty$. The first result was due to \citet{derrida04asymmetric} who investigated a restricted area within the fan region. 
The complete description of limit fluctuations of in all three phases 
except on the coexistence line
has been provided by \citep{bryc19limit} (fan region) and \citep{wang24askey} (shock region) together, and we refer to \citep{wang24askey} for more details. 
 \end{remark}

\begin{remark}\label{rem:ABCD}With $A_n\to 1$ and $C_n\to 1$ at appropriate rates and $q\in[0,1)$ fixed, the limit theorem in \citep{bryc23asymmetric} in the fan region also allows $\A = \infty$ and/or $\C = \infty$. The process in the limit $\eta\topp{\A,\infty}$ is a randomized Brownian meander when $\A\in\R$ (a standard Brownian meander when $\A = 0$) and a Brownian excursion when $\A = \infty$. For our case, with $\A+\C<0$, we choose to not include the case that $\A$ and/or $\C = -\infty$.
In these cases, although the centering in the definition of the height function $h_n$ given by \eqref{eq:h_n} may have to be modified accordingly for the limit theorem to hold, we expect  
the limit fluctuation to be a
 Brownian motion, as long as the parameters remain off the coexistence line. More precisely, we expect the right-hand side of \eqref{eq:fdd} to be (a scalar multiple of) a Brownian motion; in particular, the process $\eta$ does not appear in the limit.  

The constraints that $B_n,D_n, B,D>-1$ is for convenience only. The parameters $B_n$ and $D_n$ were allowed to approach $-1$ (with controlled rates of convergence) in the studies in \citep{bryc23asymmetric} of fan region, and the proof remained essentially the same except that there were a few more terms to be taken care of in the intermediate steps only. It is well understood that the parameters $B_n$ and $D_n$ do not have any influence in the limit. Our result remains to hold with the same limits, if $B,D=-1$ are allowed.
\end{remark}

\begin{remark}
When allowing all the five parameters $A,B,C,D,q$ to change appropriately as $n\to\infty$, and in particular, when $A_n\to1, C_n\to 1$, and $q_n\to 1$ at appropriate rates, it is proven by \cite[Theorem 1.2 (1)]{corwin24stationary} (which builds on earlier works \cite{corwin18open,parekh19KPZ}) that the limiting stationary measures of open ASEP become the stationary measures of   the {\em open KPZ equation with Neumann type boundary conditions}. The open KPZ equation, introduced by \citet{corwin18open}, serves as the counterpart to the KPZ equation introduced by \citet{hairer13solving} (defined on the full line) when restricted to an interval. 
This limit regime is technically the most involved case, and again only the fan region has been studied. The groundbreaking work in this setup is due to \citet{corwin24stationary} who first computed the limit Laplace transform of stationary measures. Their results were later complemented by \citet{barraquand22steady} and \citet{bryc23markov} who provided more accessible descriptions of the stationary measures (see also \citet{bryc22markov} and a nice survey paper by \citet{corwin22some}). Most recently, \citet{himwich24convergence} showed that the assumptions in \citep{corwin24stationary} can be relaxed (but always restricted to the fan region).
We believe that the method developed in \citep{wang24askey} can be adapted to attack the shock region in this setup. This is left for a future study.  
\end{remark}

\subsection{Outline of the rest of the paper} 
Section \ref{sec:AWP}  reviews Askey--Wilson signed measures and the integral representation of open ASEP. 
Section \ref{sec:eta} presents two alternative representations of $\eta\topp{\A,\C}$, along with the fact that its Laplace transform is continuous in $\A,\C$. Section \ref{sec:proof} provides the proof of Theorem \ref{thm:1}.

\section{Askey--Wilson signed measures and open ASEP}\label{sec:AWP}
In this section we review the background
 of the Askey--Wilson signed measures, 
and their role in the analysis of stationary measures of open ASEP.
More details can be found in \citep{bryc10askey,bryc17asymmetric,wang24askey}. 

Askey--Wilson signed measures depend on five parameters $a,b,c,d,q$, under the following assumptions. 
\begin{assumption}\label{assump:AW}
Suppose $q\in[0,1)$, and $a,b,c,d\in\mathbb C$ are such that
\begin{enumerate}[(i)]
\item $abcd\notin\{q^{-\ell}:\ell\in\N_0\}$, where we denote $\N_0=\{0,1,2,\dots\}$,
\item for any two distinct $\mathfrak e,\mathfrak f\in\{a,b,c,d\}$ such that $|\mathfrak e|, |\mathfrak f|\ge 1$, we have $\mathfrak e/\mathfrak f\notin\{q^\ell:\ell\in\Z\}$,
\item $a,b$ are real, and $c,d$ are either real or form a complex conjugate pair, and $ab<1, cd<1$. \end{enumerate}
\end{assumption}
We write 
\[
(a,b,c,d)\in\Omega_q ,
\]
 if $a,b,c,d,q$ satisfy the above assumption. In such a case, we let $\nu(\d y;a,b,c,d,q)$ denote the Askey--Wilson signed measure. This is a finite signed measure of total mass one and with compact support in $\R$, which is of mixed type
 \[
\nu(\d y;a,b,c,d,q)=f(y;a,b,c,d,q)\d y+\sum_{z\in F(a,b,c,d,q)}p(z)\delta_z(\d y).
\]
The absolutely continuous part is supported on $[-1,1]$ and the discrete part is supported on a finite or empty set $F=F(a,b,c,d,q)$.
For certain choices of parameters, the measure can be only discrete or only continuous.
The density function of the continuous part of $\nu(\d y;a,b,c,d,q)$ is  
\begin{equation}\label{eq:f}
  f(y;a,b,c,d,q)=\frac{\ds\qp{q,ab,ac,ad,bc,bd,cd}}{\ds2\pi\qp{abcd}\sqrt{1-y^2}}\,\left|\frac{\ds\qp{e^{2i\theta_y}}}
{\ds\qp{ae^{i\theta_y},be^{i\theta_y},ce^{i\theta_y},de^{i\theta_y}}}\right|^2\inddd{|y|< 1},
\end{equation}
where $\theta_y$ is such that $y=\cos\,\theta_y$.
Here and below, for $\alpha\in\mathbb C$,  $n\in\N_0\cup\{\infty\}$ and $|q|<1$ we
use the
$q$-Pochhammer symbol:
\begin{equation*}
\qps{\alpha}_n=\prod_{j=0}^{n-1}\,(1-\alpha q^j), \quad \qps{a_1,\cdots,a_k}_n =\prodd j1k\qps{a_j}_n.
\end{equation*}
The set $F=F(a,b,c,d,q)$  of atoms of $\nu(\d y;a,b,c,d,q)$ is   non-empty if  there is a real parameter $\mathfrak e\in\{a,b,c,d\}$  with
$|\mathfrak e|\geq 1$, and in this case $F$ contains the following atoms:
\[
  y\topp{\mathfrak e}_j=\frac12\pp{\mathfrak eq^j+\frac1{\mathfrak eq^j}} \mbox{ for $j=0,1,\dots$ such that $|\mathfrak eq^j|\ge 1$}.
\]
We refer to all these  as the {\em atoms generated by $\mathfrak e$}.

The corresponding masses on the atoms are, if $\mathfrak e = a$,
\[
    \begin{split}
        p\lb y_0\topp a;a,b,c,d,q\rb& =\frac{\qp{a^{-2},bc,bd,cd}}{\qp{b/a,c/a,d/a,abcd}},\;  
 \\
p\lb y_j\topp a;a,b,c,d,q\rb& =p(y_0;a,b,c,d,q)\frac{\ds\qps{a^2,ab,ac,ad}_j\,(1-a^2q^{2j})}{\ds\qps{q,qa/b,qa/c,qa/d}_j(1-a^2)}\left(\frac{\ds q}{\ds abcd}\right)^j,\quad
j\ge 1, 
    \end{split}
\]
and similar formulas with $a$ and $\mathfrak e$ swapped when $\mathfrak e\in\{b,c,d\}$.  Note that the possible number of parameters $a,b,c,d$ that generate atoms are $0, 1,2$ (see the 
Condition
 (iii) of Assumption \ref{assump:AW}). 
We also mention that the
expression for
$p(y_j\topp a;a,b,c,d,q)$  given here only applies for $a,b,c,d\ne 0$, and takes a different form otherwise. 
See \citep{wang24askey} for more details.

Next, with $(A,B,C,D,q)$ fixed we shall work with certain Askey--Wilson signed measures denoted by $\pi_t(\d x)$ and $P_{s,t}(x,\d y)$ for suitably chosen $s,t>0$. The choice of $s,t$ is delicate in the case of $AC>1$, which we summarize in the following lemma.
\begin{lemma}\label{lem:time interval}
    Assume $A,B,C,D,q$ satisfy \eqref{eq:conditions qABCD}. If $AC<1$, set $\mathbb T = \mathbb T(A,B,C,D,q) := (0,\infty)$. If $AC>1$, assume further $ABCD\notin\{q^{-\ell}:\ell\in\N_0\}$ and set
\[\mathbb{T}=\mathbb{T}(A,B,C,D,q):=\lb\max
    \ccbb{\sqrt q,D^2}
    ,\min
    \ccbb{\frac1{\sqrt q},\frac1{B^2}}
    \rb\setminus\ccbb{\frac CAq^{\ell}:\ell\in\Z}.
    \]
    \begin{enumerate}[(i)]
    \item
    For all $t\in\mathbb{T}$ we have 
\[
\lb A\sqrt t,B\sqrt t, \frac C{\sqrt t}, \frac D{\sqrt t}\rb  \in\Omega_q,
\]
and we set
\[ 
\pi_t(\d y) \equiv \pi_t\topp{A,B,C,D,q}(\d y):=\nu\pp{\d y;A\sqrt t,B\sqrt t, \frac C{\sqrt t}, \frac D{\sqrt t},q},\quad t\in\mathbb{T}.
\] 
Note that the support of $\pi_t$ is 
\[
\mathbb U_t\equiv \mathbb U_t(A,B,C,D,q):= [-1,1]\cup F\pp{A\sqrt t,B\sqrt t, \frac C{\sqrt t}, \frac D{\sqrt t},q}. 
\]
\item Moreover, for all $s,t\in\mathbb T$, $s<t$ and $x\in\mathbb{U}_s$, we have
\begin{equation}\label{eq:transition well defined}
\lb A\sqrt{t},B\sqrt t,\sqrt{\frac st}\lb x+\sqrt{x^2-1}\rb,\sqrt{\frac st}\lb x-\sqrt{x^2-1}\rb\rb\in\Omega_q.
\end{equation}
    \end{enumerate}

    \end{lemma} 
The case $AC<1$ is well-known; see \citep{bryc17asymmetric}. The proof of the case $AC>1$ is postponed towards the end of this section. Note that in the case $AC>1$, the set $\mathbb T$ is not the largest set possible, although it serves our purpose: in particular, 
 $\mathbb T$ contains an open interval containing $1$ when $A\ne C$ (which is key to our analysis later). It is also worth noticing that $1\notin\mathbb T$ when $A=C>1$ (and in this case we eventually need a different argument).  

The second part of the previous lemma (Equation \eqref{eq:transition well defined}) guarantees that the following Askey--Wilson signed measures are well-defined.
For $s<t$, $s,t\in\mathbb T$ and $x\in \mathbb U_s$, set
\[ 
\begin{split}
    P_{s,t}(x,\d y)& \equiv P_{s,t}\topp{A,B,C,D,q}(x,\d y)\\
&:=\nu\pp{\d y;A\sqrt{t},B\sqrt t,\sqrt{\frac st}\lb x+\sqrt{x^2-1}\rb,\sqrt{\frac st}\lb x-\sqrt{x^2-1}\rb,q},
\end{split}
\] 
and $P_{s,t}(x,\d y) := 0$ with $x\notin\mathbb U_s$. 
When $|x|<1$, expression $x\pm \sqrt{x^2-1}$ is understood as $e^{\pm \i\theta_x}$ with  $\cos\theta_x = x$. Set $P_{s,s}(x,\d y) := \delta_x(\d y)$ if $x\in\mathbb U_s$ and $P_{s,s}(x,\d y)=0$ otherwise. 
It is known from \cite[Lemma 2.14]{wang24askey} that for any $s\leq t$, $\mathbb U_t$ contains the support of $P_{s,t}(x,\d y)$.

Moreover, when $AC<1$, it is known that the Askey--Wilson signed measures $\pi_t(\d y)$ and $P_{s,t}(x,\d y)$ become probability measures, and it was shown in \citep{bryc10askey}
 that there exists a time-inhomogeneous Markov process $(Y_t)_{t>0}$ such that 
\[
\proba(Y_t\in \d y) = \pi_t(\d y) \qmand \proba(Y_t\in \d y\mid Y_s = x) = P_{s,t}(x,\d y).
\]
Then, by standard notation in Markov processes, 
we also let
\equh\label{eq:pi joint}
\pi_{t_1,\dots,t_d}(\d y_1,\dots,\d y_d) := \pi_{t_1}(\d y_1)P_{t_1,t_2}(y_1,\d y_2)\cdots P_{t_{d-1},t_d}(y_{d-1},\d y_d)
\eque
to 
denote
 the 
finite-dimensional
 distribution of the process $Y$ at times $0< t_1\le \cdots\le t_d$.
When $AC>1$, the interpretation of $\pi_t$ and $P_{s,t}$ as laws of Markov processes no longer holds since they are no longer non-negative measures. But,
for any $t_1,\dots,t_d\in\mathbb{T}$ satisfying $0< t_1\le \cdots\le t_d$, equation \eqref{eq:pi joint} remains to define a signed measure on $\R^d$, 
which
 will play a crucial role in our analysis.

At the heart of our analysis of the stationary measures of open ASEP is the following 
identity
 from \citep{bryc17asymmetric,wang24askey}.
\begin{theorem}\label{thm:AW rep}
Consider an open ASEP with parameters $\alpha,\beta>0,\gamma,\delta\ge 0$ and $q\in[0,1)$. We have seen that the parameterization $A,B,C,D$ in~\eqref{eq:defining ABCD} satisfy \eqref{eq:conditions qABCD}. We recall that 
$(\tau_{n,1},\dots,\tau_{n,n})\in\{0,1\}^n$ follows the stationary measure of open ASEP with $n$ locations, denoted by $\mu_n$. Assume $ABCD\notin\{q^{-\ell}:\ell\in\N_0\}$.
\begin{enumerate}[(i)]
\item
There exists a polynomial function $\Pi_n(t_1,\dots,t_n)$ (see Remark \ref{rem:DEHP}), such that for all $t_1,\dots,t_n>0$,
\equh\label{eq:MPA}
\esp\lee\prodd i1n t_i^{\tau_{n,i}}\ree = \frac{\Pi_n(t_1,\dots,t_n)}{\Pi_n(1,\dots,1)}.
\eque 
\item 
For $0<t_1\le\cdots\le t_n$ and $t_1,\dots,t_n\in\mathbb T$, we have the integral representation
\equh\label{eq:AW rep}
\Pi_n(t_1,\dots,t_n)=\int_{\R^n}\prodd j1n\lb1+t_j+2\sqrt{t_j}y_j\rb\pi_{t_1,\dots,t_n}(\d y_1,\dots,\d y_n).
\eque
\end{enumerate}
\end{theorem} 
\begin{proof}
The first part is due to  \citet{derrida93exact}. For the second part, 
    the result is essentially \cite[Theorem 1.1]{wang24askey}.  The proof in \citep[Section 3.3]{wang24askey} concerns a specific choice of set denoted by $I$ therein, although the same proof remains to work with $I$ replaced by $\mathbb T$ introduced here.
\end{proof}

\begin{remark}\label{rem:DEHP}
The polynomial function $\Pi_n$ is given by 
\[
\Pi_n(t_1,\dots,t_n):=    (1-q)^n\aa{ W|(\mathsf E+t_1\mathsf D)\times\dots\times(\mathsf E+t_n\mathsf D)|V}
\]
for some matrices $\mathsf D$, $\mathsf E$, a row vector $\langle W|$ and a column vector $|V\rangle$ with the same (possibly infinite) 
dimensions,
 satisfying the following relation,
    \equh\label{eq:DEHP algebra} 
        \mathsf D\mathsf E-q\mathsf E\mathsf D=\mathsf D+\mathsf E, \quad
        \langle W|(\alpha\mathsf E-\gamma\mathsf D)= \langle W|, \quad
        (\beta\mathsf D-\delta\mathsf E)|V\rangle=|V\rangle,
 \eque
 commonly referred to as the DEHP algebra, named after the authors of the seminal work \citep{derrida93exact}. The representation \eqref{eq:MPA} with \eqref{eq:DEHP algebra} is also known as the matrix product ansatz. Many concrete examples satisfying \eqref{eq:DEHP algebra} have been discovered; see for example \citep{blythe00exact,derrida93exact,essler96representations,enaud04large,mallick97finite,sandow94partially,sasamoto99one,uchiyama04asymmetric}.  
  
The integral representation of \eqref{eq:MPA} using \eqref{eq:AW rep}, when $AC<1$, was first discovered by \citet{bryc17asymmetric}, where the measure $\pi_{t_1,\dots,t_n}$ is the finite-dimensional distribution of the Askey--Wilson Markov process (with parameters $A,B,C,D,q$), denoted by $Y$, at times $t_1,\dots,t_n$. This representation utilizes the matrix product ansatz and the example of  $\mathsf D$, $\mathsf E$, $\langle W|$, $|V\rangle$ satisfying  DEHP algebra   found by \citet{uchiyama04asymmetric} (related to the Askey--Wilson orthogonal polynomials \citep{askey85some}). The identity \eqref{eq:MPA} in this case can be restated as
\[ 
\esp\lee\prod_{j=1}^n t_{j}^{\tau_{n,j}
}\ree = \frac{ \esp\lee\prod_{j=1}^n(1+t_j+2\sqrt{t_j}\,Y_{t_j})\ree}{ 2^n\esp(1+Y_1)^n}, \quad 0<t_1\leq \cdots\leq t_n.
\] 
  This identity played a fundamental role in recent analysis of stationary measures of open ASEP \citep{bryc19limit,bryc23asymmetric,corwin24stationary}.
When $AC>1$, it was recently discovered in \citep{wang24askey} that the above identity \eqref{eq:MPA} with \eqref{eq:AW rep} remains to hold, and moreover, in its application to open ASEP some estimates on $\pi_t$ and $P_{s,t}$ can be borrowed from the $AC<1$ case; the analysis does become more involved at a few places, in particular when the parameters are on/approaching the coexistence line. 
\end{remark}
\begin{remark}
What is lost when passing from $AC<1$ to $AC>1$ is that now the Askey--Wilson measures are no longer probability measures and hence there is no Markov process associated to them. Nonetheless, the relations $\int P_{s,t}(x,\d y)\pi_s(\d x) = \pi_t(\d y)$ for $s\leq t$, $s,t\in\mathbb T$ and $\int P_{r,s}(x,\d y)P_{s,t}(y,\d z) = P_{r,t}(x,\d z)$ for $r\leq s\leq t$, $r,s,t\in\mathbb T$ remain to hold as in the case of Markov processes. It is worth pointing out that when $AC<1$ we do not need any Markov properties. \end{remark}
\begin{proof}[Proof of Lemma \ref{lem:time interval}]
 We first show that for $t\in\mathbb T$ we have
    \[
    \lb A\sqrt t,B\sqrt t, \frac C{\sqrt t}, \frac D{\sqrt t}\rb\in\Omega_q.
    \]
    Condition
     (i) in Assumption \ref{assump:AW} for $\Omega_q$ follows from $abcd=ABCD\notin\{q^{-\ell}:\ell\in\N_0\}$. 
     Condition (ii) holds because $a,c>0$, $b,d<0$, $bd=BD<1$ and $c/a=C/(At)\notin\{q^{\ell}:\ell\in\Z\}$ (since $t\notin\{q^{\ell}C/A:\ell\in\Z\}$ by our choice of $\mathbb T$). 
     Condition (iii) holds trivially.

    We next show that for $s<t$ in $\mathbb{T}$ and $x\in \mathbb{U}_s$ we have
    \[
    \lb A\sqrt{t},B\sqrt t,\sqrt{\frac st}\lb x+\sqrt{x^2-1}\rb,\sqrt{\frac st}\lb x-\sqrt{x^2-1}\rb\rb\in\Omega_q.
    \]
    Condition (i) in Assumption \ref{assump:AW} for $\Omega_q$ follows from $abcd=ABs\leq0$.  
    Condition (iii) holds trivially. To prove 
    Condition (ii), we observe that since $B\sqrt{s}, D/\sqrt{s}\in(-1,0]$, any possible atom in $U_s$ is generated by either $A\sqrt{s}$ or $C/\sqrt{s}$. Note also that we always have $b=B\sqrt{t}\in(-1,0]$ for $t\in \mathbb T$. We split into the following three cases:
    \begin{enumerate}[(a)]
    \item  
    Let $x\in[-1,1]$ then $c$ and $d$ are complex conjugate pairs with norm $<1$ and  
    Condition (ii) vacuously holds.
    \item  
     Let $x=\frac{1}{2}( q^jA\sqrt{s}+\lb q^jA\sqrt{s}\rb^{-1})$ for $j\in\N$ and $q^jA\sqrt{s}>1$. Then $c=q^jAs/\sqrt{t}$ and $d=1/(q^jA\sqrt{t})<1$. We have $c/a=q^js/t$. Since $\sqrt{q}<s<t<1/\sqrt{q}$ we have $s/t\in(q,1)$ and hence $c/a=q^js/t\notin\{q^{\ell}:\ell\in\Z\}$. Therefore 
    Condition (ii) holds.
    \item  
     Let $x=\frac{1}{2}( q^jC/\sqrt{s}+\lb q^jC/\sqrt{s}\rb^{-1})$ for $j\in\N$ and $q^jC/\sqrt{s}>1$. Then $c=q^jC/\sqrt{t}$ and $d=s/(q^jC\sqrt{t})<1$. We have $c/a=q^jC/(At)$. By our assumption, $t\notin\{q^{\ell}C/A:\ell\in\Z\}$ and hence $c/a=q^jC/(At)\notin\{q^{\ell}:\ell\in\Z\}$.  
    Condition (ii) holds. 
\end{enumerate} 
We conclude the proof of the lemma.
\end{proof}

\section{Representations of
 stationary measures of open KPZ fixed point}\label{sec:eta}
 In this section, we assume $\A,\C\in\R$. 
We recall the definition \eqref{eq:RN1} of the process $\eta\topp{\A,\C}$ in Section \ref{subsec:main results}: 
\begin{equation}\label{eq:def of eta again}
\frac{\d \mathbb{P}_{\eta\topp{\A,\C}}}{\d \mathbb{P}_{\rm Bm}}(\vv\omega) = \frac{\ds e^{(\A+\C)\min_{t\in[0,1]}\omega_t - \A\omega_1}}{\ds \esp_{\rm Bm}\lee e^{(\A+\C)\min_{t\in[0,1]}\omega_t - \A\omega_1}\ree},\quad \vv\omega= \lb\omega_t \rb_{t\in[0,1]}\in C([0,1]),\end{equation}
where $\mathbb{P}_{\rm Bm}$ denotes the law on $C([0,1])$ of a standard Brownian motion, and $\mathbb{E}_{\rm Bm}$ denotes the expectation with respect to $\mathbb{P}_{\rm Bm}$. 

We provide two alternative representations below, of which we first learned from 
\citet[arXiv version]{barraquand22steady}, and fill in with more details. These results are of their own interest. Then, we show that the Laplace transform of $\eta\topp{\A,\C}$ is continuous in $\A,\C$, which will be needed later.

\subsection{A two-line representation}

The first representation is known as the two-line/two-layer representation, which can be found in \citep{barraquand22steady,barraquand24stationary,bryc24two}. 
Consider the path space $C([0,1])^2$ of two real-valued stochastic processes. Each element in this space can be denoted by $(\vv\omega\topp1,\vv\omega\topp 2)$, where $\vv\omega\topp i = (\omega_t\topp i)_{t\in[0,1]}\in C([0,1])$ is called the $i$-th line, for $ i=1,2$.  We let 
 $\mathbb{P}^{\mu_1,\mu_2}_{\rm Bm}$ denote the law on $C([0,1])^2$ of two independent Brownian motions with drifts $\mu_1$ and $\mu_2$, respectively  (a Brownian motion with drift $\mu$ is the process $(\BB_t+\mu t)_{t\ge 0}$, where $\BB$ is a standard Brownian motion).
 
\begin{proposition}
    \label{prop:two line}
Let $\mathbb{P}_{\rm TL}^{\A,\C}$ denote the probability measure on $C([0,1])^2$ determined by
\equh\label{eq:RN2}
\frac{\d\mathbb{P}_{\rm TL}^{\A,\C}}{\d \mathbb{P}^{-\A,\A}_{\rm Bm}}(\vv\omega\topp1,\vv\omega\topp2) = \frac1{\calZ_{\A,\C}} e^{(\A+\C)\min_{t\in[0,1]}(\omega_t\topp 1-\omega_t\topp 2)},\quad  (\vv\omega\topp 1,\vv\omega\topp2)\in C([0,1])^2,
\eque
where $\calZ_{\A,\C}$ is a normalizing constant. 
Then, $\mathbb{P}_{\rm TL}^{\A,\C}$ is the law of two stochastic processes 
\[\lb\frac1{\sqrt 2}\lb \eta\topp{\sqrt 2\A,\sqrt 2 \C}_t+\BB_t\rb_{t\in[0,1]}, \frac1{\sqrt 2}\lb \eta\topp{\sqrt 2\A,\sqrt 2 \C}_t-\BB_t\rb_{t\in[0,1]}\rb,\]
where the process $\eta\topp{\sqrt 2\A,\sqrt 2\C}$ and the standard Brownian motion $\BB$ are independent. In particular, the law of the limit process in Theorem \ref{thm:1} is the marginal distribution of $\mathbb{P}_{\rm TL}^{\A/2,\C/2}$ 
on the first line.
\end{proposition}
\begin{proof}
Let $\mathbb{E}_{\rm TL}^{\A,\C}$ denote the expectation with respect to $\mathbb{P}_{\rm TL}^{\A,\C}$. For an arbitrary bounded continuous function $F:C([0,1])^2\to \R$, we have:
\begin{align*}
\mathbb{E}_{\rm TL}^{\A,\C}&\lee  F(\vv\omega_1,\vv\omega_2) \ree = \frac1{\calZ_{\A,\C}}\mathbb{E}_{\rm Bm}^{-\A,\A} \lee F(\vv\omega_1,\vv\omega_2) e^{(\A+\C)\min_{t\in[0,1]}(\omega\topp 1_t-\omega\topp 2_t)}\ree\\ 
& =  \frac1{\calZ_{\A,\C}}\mathbb{E}_{\rm Bm}^{-\A,\A}\Bigg[ F\pp{\frac{(\vv\omega\topp 1-\vv\omega\topp 2)+(\vv\omega\topp1+\vv\omega\topp 2)}2,\frac{(\vv\omega\topp 1+\vv\omega\topp 2)-(\vv\omega\topp1-\vv\omega\topp 2)}2 } 
\\
& \quad\quad\times  e^{(\A+\C)\min_{t\in[0,1]}(\omega\topp 1_t-\omega\topp 2_t)}\Bigg].
\end{align*}
Notice that:
\[\lb(\BB\topp1_t-\A t)- (\BB\topp 2_t +\A t),(\BB\topp 1_t-\A t)+(\BB\topp 2_t+\A t)\rb_{t\in[0,1]} \eqd \lb\sqrt 2\BB\topp 1_t-2\A t, \sqrt 2\BB\topp 2_t\rb_{t\in[0,1]}.
\]
Therefore,
\begin{align*}
\mathbb{E}_{\rm TL}^{\A,\C} &\lee F(\vv\omega_1,\vv\omega_2)\ree 
= \frac1{\calZ_{\A,\C}}\mathbb{E}_{\rm Bm}^{-\sqrt 2\A,0} \lee F\pp{\frac{\vv\omega\topp 1+\vv\omega\topp 2}{\sqrt 2},\frac{\vv\omega\topp2-\vv\omega\topp1}{\sqrt 2}}e^{\sqrt 2(\A+\C)\min_{t\in[0,1]}\omega\topp1_t}\ree\\
& = \frac 1{\calZ_{\A,\C}}\mathbb{E}_{\rm Bm}^{0,0}\lee F\pp{\frac{\vv\omega\topp1+\vv\omega\topp2}{\sqrt 2},\frac{\vv\omega\topp2-\vv\omega\topp1}{\sqrt 2}}e^{\sqrt 2(\A+\C)\min_{t\in[0,1]}\omega\topp 1_t}e^{-\sqrt 2\A\omega\topp1_1-\A^2}\ree\\
& =\frac{\calH\lb\sqrt{2}\A,\sqrt{2}\C\rb}{e^{\A^2}\calZ_{\A,\C}}\esp\lee {F\pp{\frac1{\sqrt 2}\pp{\eta\topp{\sqrt 2\A,\sqrt 2\C}+\BB },\frac1{\sqrt 2}\pp{\eta\topp{\sqrt 2\A,\sqrt 2\C}-\BB}}}\ree,
\end{align*}
where in the second step we applied Girsanov theorem to the first line, and in the third step we used the definition \eqref{eq:def of eta again} of process $\eta\topp{\A,\C}$. Set $F=1$ we have $\calH\lb\sqrt{2}\A,\sqrt{2}\C\rb=e^{\A^2}\calZ_{\A,\C}$. We conclude the proof. 
\end{proof}

\subsection{Joint probability density function}
The second representation first showed up in \citet[equation (56), arXiv version]{barraquand22steady},
which specifies the joint probability density function of the process $\eta\topp{\A,\C}$. 
For $t>0$, let
\[
\q_t(x,y) := \frac1{\sqrt {2\pi t}}\bb{\exp\pp{-\frac 1{2t}(x-y)^2} - \exp\pp{-\frac 1{2t}(x+y)^2}}\inddd{x,y>0},  
\]
denote the transition kernel of the Brownian motion killed at hitting zero, and write
\begin{align*}
\q\topp b_t(x,y)  
&:= \q_t(x-b,y-b) \\
& = \frac1{\sqrt{2\pi t}}\bb{\exp\pp{-\frac1{2t}(x-y)^2} - \exp\pp{-\frac1{2t}(x+y-2b)^2}}\inddd{x,y>b}.
\end{align*}

\begin{proposition}
    \label{prop:BLD}
The stochastic process $(\eta\topp{\A,\C}_t)_{t\in[0,1]}$ has joint probability density function at times $t_1,\dots,t_d$, with $0=t_0<t_1<\cdots<t_d = 1$, given by
\equh\label{eq:BLD}
p\topp{\A,\C}_{t_1,\dots,t_d}(z_1,\dots,z_d) = \frac1{\calH\lb\A,\C\rb} \int_{-\infty}^{\min\limits_{k=0,\dots,d}z_k} e^{(\A+\C)b-\A z_d} (-\partial b)\pp{\prodd k1d \q_{\Delta t_k}\topp b(z_{k-1},z_k)}\d b, 
\hspace{.1cm}\vv z\in\R^d,
\eque
with $z_0 = 0$ and $\Delta t_k=t_k-t_{k-1}$ for $k=1,\dots,d$.
\end{proposition}
\begin{proof}
Introduce
\[
\q^{(b),*}_t(x,y):=(- \partial b) \q_t\topp b(x,y) = \frac{\ds \sqrt 2}{\ds \sqrt{\pi t^3}}(x+y-2b)\exp \pp{-\frac{\ds (x+y-2b)^2}{\ds 2t}}\inddd{x,y>b},   \quad t>0.
\]
Notice that $\q^{(b),*}_t(x,y)>0$ for any $x,y,b\in\R$ and $t>0$.
By the reflection principle, we have:
\[ 
    \begin{split}
        \proba\pp{\min_{s\in[0,t]}\mathbb B_s\in \d b, \mathbb B_t\in \d y\mmid \mathbb B_0 = x} & 
= \proba\pp{\min_{s\in[0,t]}\mathbb B_s\in \d (b-x), \mathbb B_t\in \d (y-x)\mmid \mathbb B_0 = 0}\\
& = \proba\pp{\max_{s\in[0,t]}\mathbb B_s\in \d (x-b), \mathbb B_t\in \d (x-y)\mmid \mathbb B_0 = 0} \\
& = \frac{\ds 2(2(x-b)-(x-y))e^{-(2(x-b)-(x-y))^2/(2t)}}{\ds \sqrt{2\pi t^3}}\inddd{x,y>b} \\& = \q_t^{(b),*}(x,y)\d b\d y.
    \end{split}
\] 
By the above and the Markov property, we have
\begin{equation}\label{eq:vaeraw}
\proba\pp{\BB_{t_i}\in \d z_i,\min_{t\in(t_{i-1},t_i]}\BB_t \in \d b_i, i=1,\dots,d}= \prodd i1d \pp{\q_{\Delta t_i}^{(b_i),*}(z_{i-1},z_i) }  \d \vv z\d \vv b.
\end{equation}
Now fix $j\in\{1,\dots,d\}$ and write $b_j = b$. Notice that for each $i\ne j$,
\[
\int_{b}^{\min\lb z_{i-1},z_i\rb}\q_{\Delta t_i}^{(b_i),*}(z_{i-1},z_i)\d b_i = \q\topp{b}_{\Delta t_i}(z_{i-1},z_i). 
\]
Integrate \eqref{eq:vaeraw} over all $i\in\{1,\dots,d\}$ and $i\neq j$, we have:
\begin{multline}\label{eq:Shepp1}
\proba\pp{\BB_{t_i}\in \d z_i, i=1,\dots,d,\min_{t\in(t_{j-1},t_j]}\BB_t \in \d b, \min_{t\in(t_{j-1},t_j]}\BB_t = \min_{t\in[0,1]}\BB_t}\\
= \prod_{\substack{i=1,\dots, d\\i\ne j}}\q\topp b_{\Delta t_i}(z_{i-1},z_i) \times \q_{\Delta t_j }^{(b),*}(z_{j-1},z_j) \d \vv z\d b.
\end{multline}
This identity can be regarded as an extension of reflection principle, which is closely related to Shepp's formula \citep{shepp79joint} and  Denisov's decomposition of Brownian motion \citep[Corollary 5]{pitman99brownian}. Summing over $j$, we have: 
\[ 
\proba\pp{\BB_{t_i}\in \d z_i, i=1,\dots,d,  \min_{t\in[0,1]}\BB_t\in\d b}\\
= (-\partial b)\pp{ \prodd i1d\q\topp b_{\Delta t_i }(z_{i-1},z_i)} \d \vv z\d b.
\] 
Then \eqref{eq:BLD} follows from the definition \eqref{eq:def of eta again} of process $\eta\topp{\A,\C}$. We conclude the proof.
\end{proof}
\begin{remark}
When restricted to the fan region $\A+\C>0$, the same process showed up in \citet{bryc23asymmetric}, presented as the increment process of a randomized Brownian meander  $(\wt\eta_t\topp{\A,\C})_{t\in[0,1]}$, namely
\[ 
\lb\eta_t\topp{\A,\C}\rb_{t\in[0,1]} = \lb\wt\eta_t\topp{\A,\C} - \wt\eta_0\topp{\A,\C}\rb_{ t\in[0,1]}.
\] 
 The randomized meander $\wt\eta\topp{\A,\C}$ is defined via the joint probability density function of its finite-dimensional distributions given by, for $0=t_0<\cdots<t_d=1$,
\equh\label{eq:wt p}
\wt p\topp{\A,\C}_{t_0,\dots,t_d}(x_0,\dots,x_d) = \frac{ \A+\C}{ \calH\lb\A,\C\rb}e^{-\C x_0} \prodd k1d \q_{\Delta t_k }(x_{k-1},x_k)e^{-\A x_d},\quad x_0,\dots,x_d>0.
\eque
(In the notation of \citep{bryc23asymmetric}, $\calH\lb\A,\C\rb = (\A+\C)\mathfrak C_{\sqrt 2\A,\sqrt 2\C}$).
Note that $\A+\C>0$ is important, otherwise the right-hand side above is a non-positive function. We next justify that this representation is equivalent to Proposition \ref{prop:BLD}.

From \eqref{eq:wt p} we derive the joint density function of $\eta\topp{\A,\C}$ at times $t_1,\dots,t_d$, with $0<t_1<\cdots<t_d = 1$:
\[
 p\topp{\A,\C}_{t_1,\dots,t_d}(z_1,\dots,z_d) = \int_0^\infty \wt p\topp{\A,\C}_{t_0,\dots,t_d}(x_0,x_0+z_1,\dots,x_0+z_d)\inddd{x_0+z_i>0, i=1,\dots,d}\d x_0.
 \]
 Then, using $\q_t(x+y,x+z) = \q_t\topp{-x}(y,z)$ and setting $x_0 = -b$, we have:
 \equh
 \label{eq:BWW}
p\topp{\A,\C}_{t_1,\dots,t_d}(z_1,\dots,z_d) =\frac{\A+\C}{\calH\lb\A,\C\rb}\int_{-\infty}^{\min\limits_{k=1,\dots,d}z_k}e^{(\A+\C)b-\A z_d} \prodd k1d \q\topp b_{\Delta t_k}(z_{k-1},z_k)\d b, \quad \A+\C>0.
\eque
By our assumption $\A+\C>0$, \eqref{eq:BWW} is equivalent to \eqref{eq:BLD} by integration by parts. 
\end{remark}

\subsection{A continuity result on the Laplace transform} 
We will need the following result.
\begin{proposition}\label{prop:continuity of Laplace}
    For any fixed $\vvc=(c_1,\dots,c_d)\in\R^d$ and $\vv t = (t_1,\dots,t_d)$ with $0<t_1<\cdots<t_d= 1$, the Laplace transform 
\[
\psi_\vvt^{(\A,\C)}(\vvc):=\esp\lee\exp\pp{-\summ k1d c_k\eta\topp{\A,\C}_{t_k}}\ree 
\]
 is a continuous function in $\A,\C\in\R$.
\end{proposition}
\begin{proof}
Following \eqref{eq:BLD}, by a simple change of variables we only need to show that, 
\begin{multline}\label{eq:vereav}
    \psi_\vvt\topp{\A,\C}(\vvc)
    =\frac{1}{\calH(\A,\C)}\int_{\R^{d}}\int_{-\infty}^{\min\limits_{k=0,\dots,d}z_k}\exp\lb-\sum_{i=1}^dc_iz_i+(\A+\C)b-\A z_d\rb\times \\
    (-\partial b)\lb\prod_{k=1}^d\q_{\Delta t_k}\topp{b}(z_{k-1},z_k)\rb\d b\d z_1\dots\d z_d
\end{multline}
is continuous with respect to $\A,\C\in\R$. 
Since $\calH(\A,\C)$ is a continuous function in $\A$ and $\C$, we only need to show that the integral on the right-hand side is also continuous.
      In view of the fact that $\q_t^{(b)}(x,y)\neq0$ only if $x,y>b$, by a change of variables 
      \[y_0=-b,\quad y_k=z_k-b,\quad k=1,\dots,d,\]
      the integral on the right-hand side of 
      \eqref{eq:vereav}
       is equal to:
\begin{multline}\label{eq:to prove}
    \sum_{k=1}^d\int_{\R^{d+1}}\exp\lb(c_1+\dots+c_d-\C)y_0-c_1y_1-\dots-c_{d-1}y_{d-1}-(c_d+\A)y_d\rb\times\\
    \q_{\Delta t_1}(y_0,y_1) 
    \cdots
       \q^*_{\Delta t_k}(y_{k-1},y_k)   \cdots 
    \q_{\Delta t_d}(y_{d-1},y_d)\d y_0\dots\d y_d,
\end{multline}
where we recall 
\[
\q_t(x,y) = \frac1{\sqrt {2\pi t}}\bb{\exp\pp{-\frac 1{2t}(x-y)^2} - \exp\pp{-\frac 1{2t}(x+y)^2}}\inddd{x,y>0}, \quad t>0,
\]
and we denote 
\[
\q^{*}_t(x,y):= (- \partial b)|_{b=0} \q_t\topp b(x,y) = \frac{\ds \sqrt 2}{\ds \sqrt{\pi t^3}}(x+y)\exp \pp{-\frac{\ds (x+y)^2}{\ds 2t}}\inddd{x,y>0}, \quad t>0.
\]       
By the dominated convergence theorem, we only need to show that each summand of \eqref{eq:to prove} is finite for any $c_1,\dots,c_d\in\R$. Set
\[
\upsilon_0=-c_1-\dots-c_d+\C,\quad \upsilon_j=c_j, \text{ } j=1,\dots,d-1,\quad \upsilon_d=c_d+\A.
\]
We only need to show that for each $k=1,\dots,d$ and any   $\upsilon_0,\dots,\upsilon_{d}\in\mathbb{R}$,
\begin{equation}\label{eq:vare}
    \int_{\R^{d+1}}\exp\lb-\summ r0d\upsilon_ry_r\rb\q_{\Delta t_1}(y_0,y_1)
    \cdots   \q^*_{\Delta t_k}(y_{k-1},y_k)   
    \cdots \q_{\Delta t_d}(y_{d-1},y_d)\d y_0\dots\d y_d <\infty.\end{equation} 
Notice that:
\[
\q_t(x,y)\leq\frac1{\sqrt {2\pi t}}\bb{\exp\pp{-\frac 1{2t}(x-y)^2} }\inddd{x,y>0},
\]
and that:
\[
\q^{*}_t(x,y)\leq C_t\exp \pp{-\frac{(x+y)^2}{4t}}\inddd{x,y>0}, 
\]
where for each $t>0$,
\[C_t:=\frac{\sqrt 2}{\sqrt{\pi t^3}}\sup_{z>0}\lb z\exp\lb-\frac{z^2}{4t}\rb\rb\in(0,\infty).\]
Therefore, the left-hand side of \eqref{eq:vare} is bounded from above by a finite constant (involving $\Delta t_1,\dots,\Delta t_d$) times:
\begin{multline*} 
    \int_{\R_+^{k}}\exp\lb-\summ r0{k-1}\upsilon_ry_r-\summ r1{k-1}\frac{(y_{r-1}-y_r)^2}{2\Delta t_r}-\frac{y_{k-1}^2}{4\Delta t_k}\rb\d y_0\dots\d y_{k-1}\\ 
\times    \int_{\R_+^{d+1-k}}\exp\lb-\summ rkd\upsilon_ry_r-\frac{y_{k}^2}{4\Delta t_k} - \summ r{k+1}d-\frac{(y_{r-1}-y_r)^2}{2\Delta t_{r}}\rb\d y_k\dots\d y_d,
\end{multline*}
which is a finite number. We conclude the proof.
\end{proof}

\section{Proof of the main theorem}\label{sec:proof}
We start with some notations and an overview of the proof.
Recall that we consider a sequence of open ASEP, each with $n$ sites and with parameters $A_n, B_n, C_n, D_n,q$ satisfying Assumption \ref{assump:0}. 
Recall the definition of $h_n(x)$ in \eqref{eq:h_n}.  
Our goal is to prove: 
\[ 
\frac1{\sqrt n}\lb h_n(x)\rb_{x\in[0,1]}
\weakto \frac1{\sqrt 2}\lb\BB_x + \eta\topp{\A/\sqrt 2,\C/\sqrt 2}_x\rb_{x\in[0,1]} \mbox{ in $D([0,1])$.}
\]
To prove the tightness, it suffices to slightly modify the argument of \citep[Proposition 2.4]{bryc23asymmetric}, and we provide a brief summary below right before Section \ref{sec:tangent}. For the rest of this section, we focus on the proof of the convergence of finite-dimensional distributions by computing the limit Laplace transform.

We fix $x_0 = 0$, and write $\vv x = (x_1,\dots,x_d)$ with $0<x_1<\cdots<x_d= 1$.  Consider the Laplace transform:
\[ 
\varphi_{n,\vv x}(\vv c) = \varphi_{n,\vv x}^{(A_n,B_n,C_n,D_n,q)}(\vv c) := 
\esp
\lee\exp\pp{-\summ k1d c_kh_n(x_k)}\ree, \quad \vv c=(c_1,\dots,c_d) \in \R^d.
\]

By \cite[Theorem A.1]{bryc19limit}   we only need to prove
\[
\limn\varphi_{n,\vvx}\pp{\frac{\vvc}{\sqrt n}} = \varphi_\vvx(\vvc), 
\]
for $\vvc = (c_1,\dots,c_d)$ in some open subset of $\R^d$, where  
\begin{equation}\label{eq:Laplace right-hand side}
\varphi_\vvx(\vvc)=\varphi_\vvx^{(\A,\C)}(\vvc):=\esp\lee\exp\pp{-\frac1{\sqrt 2}\summ k1d c_k\lb\BB_x +\eta\topp{\A/\sqrt 2,\C/\sqrt 2}_{x_k}\rb}\ree 
\end{equation}
is the Laplace transform of the desired limit process.

Recall the Askey--Wilson signed measures reviewed in Section \ref{sec:AWP}. We write 
\begin{equation}\label{eq:measures n}
\pi_t\topp n(\d y):= \pi_t\topp{A_n,B_n,C_n,D_n,q}(\d y) \qmand P_{s,t}\topp n(x,\d y) := P_{s,t}\topp{A_n,B_n,C_n,D_n,q}(x,\d y), 
\end{equation}
and use the notation $\pi_{t_1,\dots,t_d}\topp n(\d y_1,\dots,\d y_d)$ (see \eqref{eq:pi joint}) accordingly. 
Note that for these signed measures to be well-defined, it suffices to have $s\leq t$ and $t_1\leq\dots\leq t_d$ from  (recall Lemma \ref{lem:time interval})
\[ 
 \mathbb{T}_n=\mathbb{T}\lb A_n,B_n,C_n,D_n,q\rb=\lb\max\ccbb{ \sqrt{q},D_n^2},\min\ccbb{ \frac1{\sqrt{q}},\frac1{B_n^2}}\rb\setminus\ccbb{\frac{C_n}{A_n}q^{\ell}:\ell\in\Z}.
\] 

Write $s_k = c_k+\cdots+c_d$ for $k=1,\dots, d$ and $n_k = \floor{n x_k}$ for $k=0,\dots,d$. 
The starting point of the proof is a representation of the Laplace transform as a multiple integral in the next lemma.
\begin{lemma}\label{lem:starting point}
Assume $\A\ne\C$,
$c_1,\dots,c_{d-1}>0$, and
\equh\label{eq:s_k constraint}
     s_k\neq\frac{1}{2}(\C-\A),\quad  k=1,\dots,d.
\eque
Then, for $n$ large enough 
we have
\equh\label{eq:Markov rep}
\varphi_{n,\vvx}\pp{\frac\vvc{\sqrt n}}  = 
\frac{\ds \Phi_d\topp n(\vvx, \vvc)}{\ds Z_n},
\eque
with 
\begin{equation}\label{eq:integral numerator}
   \Phi_{d}\topp n(\vvx,\vvc)= 2^{-n}\int_\Rd \prodd k1d \pp{\cosh\pp{\frac{s_k}{\sqrt n}}+y_k}^{n_k-n_{k-1}}\pi\topp n_{e^{-2s_1/\sqrt n},\dots,e^{-2s_d/\sqrt n}}(\d y_1,\dots,\d y_d).
\end{equation}
and
\begin{equation}\label{eq:denominator}
Z_n = 2^{-n}\int_\R(1+y)^n\pi_1\topp n(\d y).
\end{equation}
In particular, for $n$ large enough, $e^{-2s_1/\sqrt n},\dots,e^{-2s_d/\sqrt n},1\in \mathbb T_n$.
\end{lemma}
\begin{proof}
In view of Assumption \ref{assump:0} we have $A_nB_nC_nD_n\rightarrow BD<1$, in particular $A_nB_nC_nD_n\notin\{q^{-\ell}:\ell\in\N_0\}$ for large enough $n$. Using the matrix product ansatz \eqref{eq:MPA} we have:
\begin{align*}
\varphi_{n,\vv x}(\vv c)
&=
\esp
\lee \exp\pp{-\summ k1d\sum_{j=n_{k-1}+1}^{n_k}s_k\pp{2\tau_{n,j}-1} } \ree\\
& =  \exp\pp{\summ k1d  s_k(n_k-n_{k-1})}
\esp
\lee \prod_{k=1}^{d}\prod_{j=n_{k-1}+1}^{n_k}\lb e^{-2 s_k}\rb^{\tau_{n,j}} \ree\\
& =  
e^{\summ k1d  s_k(n_k-n_{k-1})}
\frac{\ds 1}{\ds \Pi_n(1,\dots,1)} \Pi_n\lb\underbrace{e^{-2s_1},\dots,e^{-2s_1}}_{n_1},\dots,\underbrace{e^{-2s_d},\dots,e^{-2s_d}}_{n_d-n_{d-1}}\rb. 
\end{align*} 
So, \eqref{eq:Markov rep} holds
with
\begin{multline*}
    \Phi_{d}\topp n(\vvx,\vvc) := 2^{-2n}\exp\pp{\summ k1d  \frac{s_k}{\sqrt{n}}(n_k-n_{k-1})}\times\\
    \Pi_n\lb\underbrace{e^{-2s_1/\sqrt n},\dots,e^{-2s_1/\sqrt n}}_{n_1},\dots,\underbrace{e^{-2s_d/\sqrt n},\dots,e^{-2s_d/\sqrt n}}_{n_d-n_{d-1}}\rb 
\end{multline*}
and 
\[
Z_n = 2^{-2n}\Pi_n(1,\dots,1). 
\]

In order to express $\Phi_{d}\topp n(\vvx,\vvc)$ as an integral of the form \eqref{eq:AW rep} (which implies \eqref{eq:integral numerator}), we need:
\begin{equation}\label{eq:condition in interval}
 e^{-2s_k/\sqrt n}\in \mathbb T_n,\quad  k=1,\dots,d.
\end{equation}
In view of Assumption \ref{assump:0}  we have, as $n\rightarrow\infty$,  
\begin{equation}\label{eq:assumption implies ABCDn}
1/B_n^2\rightarrow1/B^2>1,\quad D_n^2\rightarrow D^2<1,\quad C_n/A_n\rightarrow1 \hspace{0.5em} \mbox{ and } \hspace{0.5em}\sqrt{n}(1-C_n/A_n)\rightarrow \C-\A.
\end{equation}
Therefore, under \eqref{eq:s_k constraint}
we have
\[\limn\sqrt{n}\lb e^{-2s_k/\sqrt n}-C_n/A_n\rb=\C-\A-2s_k\neq0,\quad  k=1,\dots,d,\]
hence for large enough $n$, condition \eqref{eq:condition in interval} holds, and \eqref{eq:integral numerator} follows. For \eqref{eq:denominator}, it suffices to observe that in view of \eqref{eq:assumption implies ABCDn}, the assumption $\A\ne\C$ implies that for large enough $n$ we have $1\in\mathbb T_n$.

Note that $\A\ne\C$ is needed only for \eqref{eq:denominator}, and $c_1,\dots,c_{d-1}>0$ and \eqref{eq:s_k constraint} is needed only for \eqref{eq:integral numerator}.
\end{proof}

{
 The proof of 
convergence of finite-dimensional distributions in 
 Theorem \ref{thm:1} consists of the following three steps.} 
\begin{enumerate}[(i)]
\item The first step is to show that
\[
\limn\varphi_{n,\vvx}\pp{\frac{\vvc}{\sqrt n}} = \varphi_\vvx(\vvc), 
\]
for $\vvc = (c_1,\dots,c_d)$ in some open subset of $\R^d$ (where the right-hand side is in \eqref{eq:Laplace right-hand side}). 
We shall exploit the integral representation in Lemma \ref{lem:starting point}; see 
 Proposition \ref{prop:limit Laplace} in Section \ref{sec:limit Laplace}. Some preparation is provided in Section \ref{sec:tangent}. 
Note that to use the integral representations we restrict ourselves first to the case that $\A\ne\C$ and $\A,\C\ne 0$. The restriction $\mathsf{a},\mathsf{c}\neq0$ will be introduced later in Lemma \ref{lem:Zn} to simplify the study of the asymptotics
of the denominator $Z_n$ given by \eqref{eq:denominator}.  
\item 
It turns out that the so-obtained limit function of the left-hand side is not immediately recognized as $\varphi_\vvx(\vvc)$. 
In other examples analyzed by the same method, the next step consists of establishing the so-called duality formula for Laplace transforms of Markov processes (see \citep{bryc19limit,bryc23asymmetric,bryc23markov}) 
in Proposition \ref{prop:duality}.
For this proposition we provide a soft argument thanks to the recent result by \citet{bryc24two}. 
Some further comments are collected in Appendix \ref{sec:duality}.
\item In the final step,  in Section \ref{subsec:proof for a=c} 
we get rid of the assumption $\A\ne\C$ and $\A,\C\ne 0$ utilizing a continuity argument. 
\end{enumerate} 
Most of the effort will be devoted to carry out the above three steps, in the following three subsections respectively. The proof of tightness can in fact be read from earlier results and is explained below.
\begin{proof}[Proof of tightness]
The tightness can be established by a coupling argument, as demonstrated in \citep[Propositions 2.3 and 2.4]{bryc23asymmetric}. It was stated for fan region only therein, but the argument applies to the shock region almost immediately. 

The coupling idea has been used earlier in \citep[Lemma 2.1]{gantert23mixing} and \citep[Lemma 4.1]{corwin24stationary}. In words, for every $n$ it is possible to construct two sequences of i.i.d.~Bernoulli random variables $\{\tau_{n,j}\topp i\}_{j=1,\dots,n}, i=1,2$ along with the $\{\tau_{n,j}\}_{j=1,\dots,n}$ of our interest on the same probability space such that
\[
\tau_{n,j}\topp1\le \tau_{n,j}\le\tau_{n,j}\topp 2, \quad j=1,\dots,n.
\]
Then, by Donsker's theorem we know that the partial-sum process of each sequence has tightness. 
Furthermore, because of the assumption $A_n\to 1, C_n\to 1$ (regardless of the fan/shock region), we can pick the two sequences be such that $\mu\topp i_n:=\esp \tau_{n,1}\topp i\to 1/2$ as $n\to\infty$ for both $i=1,2$. The tightness of the normalized height function based on $\{\tau_{n,j}\}_{j=1,\dots,n}$ then follows. 

More specifically, to modify the proof of \citep[Proposition 2.4]{bryc23asymmetric} for the shock region, it suffices to notice that by the same construction of $\{\wt\tau_j,\what\tau_j\}_{j=1,\dots,n}$ therein, instead of $\wt\tau_j\le \tau_j\le\what\tau_j$ ((2.10) therein) we have $\what\tau_j\le \tau_j\le \wt\tau_j$. The relation (2.14) therein needs to be modified accordingly, but eventually we need to establish the same inequality (2.15), the proof of which remains unchanged.
\end{proof}
\subsection{Asymptotics of Askey--Wilson signed measures}\label{sec:tangent}
The limits of the integrals \eqref{eq:integral numerator} and \eqref{eq:denominator} with respect to Askey--Wilson signed measures $\pi\topp n_{e^{-2s_1/\sqrt n},\dots,e^{-2s_d/\sqrt n}}(\d y_1,\dots,\d y_d)$ and $\pi_1\topp n(\d y)$ are determined by the integrals near the upper boundary $1$ of the continuous domain $[-1,1]$ 
(see Remark \ref{rem:fan} for a further explanation). We shall need the asymptotic behavior of the following signed measures
\begin{align}
\wt \pi_t\topp n(\d u) &:= \pi\topp n_{e^{2t/\sqrt n}}\pp{\d \pp{1-\frac u{2n}}}, \label{eq:mar}\\
\wt P_{s,t}\topp n(u,\d v)&:=P_{e^{2s/\sqrt n},e^{2t/\sqrt n}}\topp n\pp{1-\frac u{2n},\d\pp{1-\frac v{2n}}},\label{eq:tran}
\end{align}
and we write accordingly $\wt\pi_{t_1,\dots,t_d}\topp n(\d u_1,\dots,\d u_d)$ (recall \eqref{eq:pi joint}).
We introduce a few notations useful in studying the asymptotics. The signed measure $\wt\pi_t\topp n$ has continuous and discrete parts. The density function (of continuous part) and the mass function (of discrete part) are denoted  by $\wt\pi_t^{(n),\rmc}$ and $\wt\pi_t^{(n),\rmd}$ 
respectively.
 For $\wt P_{s,t}\topp n$, it again has two parts. We let $\wt P_{s,t}^{(n),\rmc,\rmc}(u, v)$ and $\wt P_{s,t}^{(n),\rmc,\rmd}(u, v)$ denote the density function and mass function respectively when $u\in[0,4n]$, which is the support of continuous part of $\wt P_{s,t}\topp n$. When $u\notin[0,4n]$ we let $\wt P_{s,t}^{(n),\rmc,\rmc}(u, v)=\wt P_{s,t}^{(n),\rmc,\rmd}(u, v)=0$. We use similar notations $\wt P_{s,t}^{(n),\rmd,\rmc}(u, v)$ and $\wt P_{s,t}^{(n),\rmd,\rmd}(u, v)$ when $u$ is an atom of $\wt \pi_s\topp n$, and when $u$ is not an atom  we let $\wt P_{s,t}^{(n),\rmd,\rmc}(u, v)=\wt P_{s,t}^{(n),\rmd,\rmd}(u, v)=0$.

We first examine the situations of atoms in the Askey--Wilson signed measures 
\[\pi\topp n_{e^{2t/\sqrt n}}(\d y)=\nu\lb\d y;A_ne^{t/\sqrt{n}},B_ne^{t/\sqrt{n}},C_ne^{-t/\sqrt{n}},D_ne^{-t/\sqrt{n}},q\rb\] 
for fixed $t\in\R$ and 
$n$ large enough.
 For the following discussions 
we write
\equh\label{eq:t_n}
t_n = e^{2t/\sqrt n},\quad n\in\N, \quad t\in\R.
\eque  
For $t\ne(\A-\C)/2$
and $n$ large enough
we have $t_n\in\mathbb T_n$ (see Lemma \ref{lem:starting point}), hence $\pi\topp n_{t_n}$ is well-defined. 
Recall that  $\A<-\C$.
For any fixed $t\ne (\A-\C)/2$,
 when $n$ gets large enough, the signed measure $\pi\topp n_{t_n}$ always has atoms. Specifically, if $t<-\C$, $\pi\topp n_{t_n}$ has one atom 
 generated
  by $C_n/\sqrt{t_n}$, and if $t>\A$, $\pi\topp n_{t_n}$ has one atom 
  generated
   by $A_n\sqrt{t_n}$. In particular, if $t\in(\A,-\C)$, the signed measure $\pi\topp n_{t_n}$ has exactly two atoms 
  generated by $A_n\sqrt{t_n}$ and $C_n/\sqrt{t_n}$ respectively. For $t\in\{\A,-\C\}$, the situations of atoms in $\pi\topp n_{t_n}$ may be different for each $n$ even when $n$ 
  becomes
   large, and we will avoid these cases.

We shall work with 
\[
t<\A<-\C.
\]
Note that in this case, $t\ne (\A-\C)/2$ holds. 

The next three lemmas provide 
asymptotics related to $\wt\pi\topp n_t, \wt P_{s,t}\topp n$ needed in the proofs later.
\begin{lemma}\label{lem:marginal atom}
    Assume 
$t<\A<-\C$.
    Consider the signed measure $\wt\pi_{t}\topp n$ defined by \eqref{eq:mar} and \eqref{eq:measures n}. For sufficiently large $n$, $\wt\pi_{t}\topp n$ has a single atom at a location $\wt y_t\topp n$ (see \eqref{eq:loc of atom} below) with mass $\wt \pi_t^{(n),\rmd}\spp{\wt y_t\topp n}$, 
    satisfying
\[ 
\lim_{n\rightarrow\infty}\wt y_t\topp n 
=-(t+\C)^2
=:y^{*,\C}_t,
\] 
and
\[ 
\lim_{n\rightarrow\infty}\wt \pi_t^{(n),\rmd}\pp{\wt y_t\topp n}=  \frac{\ds 2(\C+t)}{\ds 2t+\C-\A}>0.
\] 
\end{lemma}
\begin{proof}
In view of the discussion below \eqref{eq:t_n}, for sufficiently large $n$, the signed measure $\pi\topp n_{t_n}$ has a single atom  
generated by $C_n/\sqrt{t_n}$, at the location
\[
y_{t_n}\topp n := \frac12\pp{\frac{C_n}{\sqrt{t_n}}+\frac{\sqrt {t_n}}{C_n}}.
\] 
Hence the signed measure $\wt\pi_{t}\topp n$ has a single atom at  
\begin{equation}\label{eq:loc of atom}
\wt y_t\topp n:=2n\pp{1-y_{t_n}\topp n}=-n \pp{\frac{C_n}{\sqrt{t_n}} + \frac{\sqrt {t_n}}{C_n}-2} \quad\text{with}\quad \lim_{n\rightarrow\infty}\wt y_t\topp n 
=-(t+\C)^2.
\end{equation}

The mass of this atom is
\begin{align*}
\wt \pi_t^{(n),\rmd}\pp{\wt y_t\topp n}& = \pi_{t_n}^{(n)}\pp{\ccbb{y_{t_n}\topp n}} = \frac{\ds \qpp{\frac{t_n}{C_n^2},A_nD_n, B_nD_n, A_nB_nt_n}}{\ds \qpp{\frac{D_n}{C_n},\frac{A_nt_n}{C_n},\frac{B_nt_n}{C_n},A_nB_nC_nD_n}} \\
&\sim \frac{\ds 1-\frac{t_n}{C_n^2}}{\ds 1-\frac{A_nt_n}{C_n}}\to \frac{\ds 2(\C+t)}{\ds 2t+\C-\A}>0, 
\end{align*}
where we have used:  
\[
    A_nD_n, \frac{D_n}{C_n}\to D\in(-1,0],\quad B_nD_n,A_nB_nC_nD_n\to BD\in[0,1),
    \]
    \[A_nB_nt_n,\frac{B_nt_n}{C_n}\to B\in(-1,0],   \left(\frac{qt_n}{C_n^2};q\right)_{\infty}, \left(\frac{qA_nt_n}{C_n};q\right)_{\infty}\rightarrow(q;q)_{\infty},
    \]
    \[
    1-\frac{t_n}{C_n^2}\sim\frac{-2(\C+t)}{\sqrt{n}}\qmand    1-\frac{A_nt_n}{C_n}\sim\frac{-(2t+\C-\A)}{\sqrt{n}}.
\]
We conclude the proof of the lemma. 
\end{proof}

\begin{lemma}\label{lem:marginal cont}
    Assume $\A+\C<0$ and $t\neq(\A-\C)/2$. For sufficiently large $n$, the continuous part of   $\wt\pi_t\topp n$ is supported over $[0,4n]$ and with 
     density function  $\wt\pi_t^{(n),\rmc}(u)$ satisfying
\begin{equation}\label{eq:lim pi c}
\limn\wt\pi_t^{(n),\rmc}(u) = \frac{\ds \A+\C}{\ds \pi} \frac{\ds \sqrt u}{\ds ((\A-t)^2+u)((\C+t)^2+u)}\leq0, \quad u\geq0.
\end{equation}
Moreover, there exist constants $N,K>0$ (possibly depend on $t$) such that for all $n>N$,
\begin{equation}\label{eq:DCT bound mar}
     \abs{\wt\pi_t^{(n),\rmc}(u)}\leq K\sqrt{u},\quad u\in[0,4n].
\end{equation}
\end{lemma}
\begin{proof}
    The continuous part of $\wt\pi_t\topp n$ has density
\[
\wt\pi_{t}^{(n),\rmc}(u) =\frac{1}{2n} f\lb1-\frac{u}{2n}; A_n\sqrt{t_n},B_n\sqrt{t_n},\frac{C_n}{\sqrt{t_n}},\frac{D_n}{\sqrt{t_n}},q\rb,\quad u\in[0,4n],
\]
with function $f$ given by \eqref{eq:f}. 
The limits \eqref{eq:lim pi c} and the bounds \eqref{eq:DCT bound mar} both follow from \citep[Lemma 4.4]{bryc23asymmetric}. Note that, although \citep[Lemma 4.4]{bryc23asymmetric} requires in its statement that $\A+\C>0$ and $-\C<t<\A$, these constraints do not, in fact, play any role in the proofs. Moreover, the uniform bound \citep[Lemma 4.4 (ii)]{bryc23asymmetric} (which is the analogue of \eqref{eq:DCT bound mar}) assumes that $u\in[0,2n]$. This constraint is only utilized in the proof to address the special cases where $B=-1$ or $D=-1$. We have avoided these cases in Assumption \ref{assump:0}; hence, this constraint can also be dropped. We conclude the proof.
\end{proof}

\begin{lemma}\label{lem:transition}
    Assume 
   that $s<t<\A<-\C$. Consider the signed measure $\wt P_{s,t}\topp n\lb u,\d v\rb$ defined by \eqref{eq:tran} and \eqref{eq:measures n}. 
    \begin{enumerate}[(i)]
    \item When $u\in[0,4n]$, for sufficiently large $n$, $\wt P_{s,t}\topp n\lb u,\d v\rb$ is a probability measure without discrete atoms. The continuous part is supported on $v\in[0,4n]$ with density $\wt P_{s,t}^{(n),\rmc,\rmc}(u,v)$ 
    satisfying
   \begin{equation}\label{eq:trans c to c}
   \limn\wt P_{s,t}^{(n),\rmc,\rmc}(u,v)  =
\p_{t-s}(u,v)\frac{\ds (\A-s)^2+u}{\ds (\A-t)^2+v}\geq0,\quad u,v\geq0,
\end{equation}
where
\begin{equation}\label{eq:def of p}
\p_{t}(u,v) = \frac{\ds 2t\sqrt v}{\ds \pi(t^4+2(u+v)t^2+ (u-v)^2)}.
\end{equation}
\item 
For sufficiently large $n$, when $u=\wt y_s\topp n$ (defined by \eqref{eq:loc of atom}) is the single atom of $\wt\pi_s\topp n$, the signed measure $\wt P_{s,t}\topp n(\wt y_s\topp n,\d v)$ has a single atom at $v=\wt y_t\topp n$, with mass $\wt P_{s,t}^{(n),\rmd,\rmd}(\wt y_s\topp n,\wt y_t\topp n)$ 
  satisfying
\begin{equation}\label{eq:trans d to d}
    \lim_{n\rightarrow\infty}\wt P_{s,t}^{(n),\rmd,\rmd}\lb\wt y_s\topp n,\wt y_t\topp n\rb=\frac{\ds (\C+t)(2s+\C-\A )}{\ds (\C+s)(2t+\C-\A )}.
\end{equation}
The continuous part of $\wt P_{s,t}\topp n(\wt y_s\topp n,\d v)$ is supported on $v\in[0,4n]$, with density function $\wt P_{s,t}^{(n),\rmd,\rmc}(\wt y_s\topp n,v)$ 
 satisfying
\begin{equation}\label{eq:trans d to c}
\limn\wt P_{s,t}^{(n),\rmd,\rmc}\lb\wt y_s\topp n,v\rb 
 = \frac{\ds \A+\C}{\ds\pi}\frac{\ds\sqrt v}{\ds((\A-t)^2+v)((\C+t)^2+v)}  \frac{\ds(\A-\C-2s)(2t-2s)}{\ds(2s-t+\C)^2+v}, v\geq0, 
\end{equation}
and that there exist constants $N,K>0$ (possibly depend on $s$ and $t$) such that for all $n>N$,
\begin{equation}\label{eq:DCT bound trans}
    \abs{\wt P_{s,t}^{(n),\rmd,\rmc}\lb\wt y_s\topp n,v\rb }\leq K\sqrt{v},\quad v\in[0,4n].
\end{equation}
\end{enumerate}
\end{lemma} 
\begin{remark}
The function $\p_{t}(u,v)$ defined by \eqref{eq:def of p}
is the transitional probability density function of the square of the radical part of 3-dimensional Cauchy process (a.k.a.~$1/2$-stable Biane process \citep{kyprianou22doob,bryc16local,biane98processes}). 
\end{remark}
\begin{proof}
We 
write
 $s_n = e^{2s/\sqrt n}$ and $t_n = e^{2t/\sqrt n}$. When $u\in[0,4n]$,
 writing $x_n = 1-u/(2n)$, we have:
\begin{align*}
    \wt P_{s,t}^{(n)}&\lb u,\d v\rb = P\topp n_{s_n,t_n}\lb x_n,1-\frac{v}{2n}\rb\\
&=\nu\pp{1-\frac{v}{2n}; A_n\sqrt{t_n}, B_n\sqrt{t_n},\sqrt{\frac{s_n}{t_n}}\lb x_n+\sqrt{x_n^2-1}\rb, \sqrt{\frac{s_n}{t_n}}\lb x_n-\sqrt{x_n^2-1}\rb,q},
\end{align*}
which does not have atoms since all the four parameters are with absolute value less than 1.  The limit \eqref{eq:trans c to c} of the continuous  density follows from \citep[Lemma 4.6]{bryc23asymmetric}. Note that \citep[Lemma 4.6]{bryc23asymmetric} requires that $\A+\C>0$. As one can observe, this constraint does not play a role in its proof.

When $u=\wt y_s\topp n$ is the single atom of $\wt\pi_s\topp n$, the limit \eqref{eq:trans d to d} of the atom mass follows from:
\begin{align*} 
\wt P_{s,t}^{(n),\rmd,\rmd}\lb\wt y_s\topp n,\wt y_t\topp n\rb
&  = \frac{\ds\qpp{\frac{t_n}{C_n^2},A_nB_nt_n, \frac{A_ns_n}{C_n}, \frac{B_ns_n}{C_n}}}{\ds\qpp{\frac{A_nt_n}{C_n},\frac{B_nt_n}{C_n},\frac{s_n}{C_n^2},A_nB_ns_n}}  
 \sim \frac{\ds\lb1-\frac{t_n}{C_n^2}\rb\lb1-\frac{A_ns_n}{C_n}\rb}{\ds\lb1-\frac{A_nt_n}{C_n}\rb\lb1-\frac{s_n}{C_n^2}\rb} \\
 &\to \frac{\ds(\C+t)(2s+\C-\A )}{\ds(\C+s)(2t+\C-\A )},
\end{align*}
where we have used: 
\[A_nB_nt_n, \frac{B_ns_n}{C_n}, \frac{B_nt_n}{C_n}, A_nB_ns_n\to B\in(-1,0],\]
\[
 1-\frac{s_n}{C_n^2}\sim\frac{-2( \C+s
)}{\sqrt{n}} 
,\quad
1-\frac{A_ns_n}{C_n}\sim\frac{-(2s+\C-\A)}{\sqrt{n}} \hspace{0.5em} \mbox{ and } \hspace{0.5em}1-\frac{A_nt_n}{C_n}\sim\frac{-(2t+\C-\A)}{\sqrt{n}}.
\]  
The limit \eqref{eq:trans d to c} of the continuous density follows from (write $z_n=1-v/(2n)$):
\begin{align*}
\wt P^{(n),\rmd,\rmc}_{s,t}  \lb\wt y_s\topp n,v\rb
&= \frac{\ds1}{\ds2n}f\pp{z_n; A_n\sqrt{t_n}, B_n\sqrt{t_n},\frac{C_n}{\sqrt{t_n}},\frac{s_n}{C_n\sqrt{t_n}},q}\\
& = \frac{\ds1}{\ds2n}\frac{\ds\lb q,A_nB_nt_n,A_nC_n,\frac{A_ns_n}{C_n},B_nC_n,\frac{B_ns_n}{C_n},\frac{s_n}{t_n};q\rb_{\infty}}{\ds2\pi \qpp{A_nB_ns_n}\sqrt{1-z_n^2}}\\
& \quad\quad\times \abs{\frac{\ds\qpp{e^{2\i\theta_{z_n}}}}{\ds\qpp{A_n\sqrt{t_n}e^{\i \theta_{z_n}},B_n\sqrt{t_n}e^{\i\theta_{z_n}},\frac{C_n}{\sqrt{t_n}}e^{\i \theta_{z_n}},\frac{s_n}{C_n\sqrt {t_n}}e^{\i\theta_{z_n}}}}}^2\\
&\sim \frac1{2n}\frac{\ds(1-A_nC_n)\lb1-\frac{A_ns_n}{C_n}\rb\lb1-\frac{s_n}{t_n}\rb}{\ds2\pi\sqrt{1-z_n^2}}\\
& \quad\quad\times 
  \frac{\ds\abs{1-e^{2\i\theta_{z_n}}}^2}{\ds\abs{1-A_n\sqrt{t_n}e^{\i\theta_{z_n}}}^2\abs{1-\frac{C_n}{\sqrt{t_n}}e^{\i\theta_{z_n}}}^2\abs{1-\frac{s_n}{C_n\sqrt{t_n}}e^{\i\theta_{z_n}}}^2}\\
&\to \frac{\ds\A+\C}{\ds\pi}\frac{\ds\sqrt v}{\ds((\A-t)^2+v)((\C+t)^2+v)}  \frac{\ds(\A-\C-2s)(2t-2s)}{\ds(2s-t+\C)^2+v},
\end{align*}
where we have used:
\[ 
A_nB_nt_n, B_nC_n,\frac{B_ns_n}{C_n}, A_nB_ns_n, B_n\sqrt{t_n}e^{\i\theta_{z_n}}\to B\in(-1,0],
\] 
\begin{equation}\label{eq:asy}1-A_nC_n\sim\frac{\A+\C}{\sqrt{n}},\quad 1-\frac{A_ns_n}{C_n}\sim\frac{-(2s+\C-\A)}{\sqrt{n}},\quad1-\frac{s_n}{t_n}\sim\frac{2(t-s)}{\sqrt{n}}, \end{equation} 
\begin{multline}\label{eq:asy2}
   \abs{1-e^{2\i\theta_{z_n}}}^2 = 4(1-z_n^2),\quad \sqrt{1-z_n^2}\sim\sqrt{\frac{v}{n}},\quad  \abs{1-A_n\sqrt{t_n}e^{\i\theta_{z_n}}}^2\sim \frac{(\A-t)^2+v}{n},\\ \abs{1-\frac{C_n}{\sqrt{t_n}}e^{\i\theta_{z_n}}}^2\sim \frac{(\C+t)^2+v}{n}\hspace{0.5em} \mbox{ and } \hspace{0.5em} \abs{1-\frac{s_n}{C_n\sqrt{t_n}}e^{\i\theta_{z_n}}}^2\sim\frac{(2s-t+\C)^2+v}{n}.
\end{multline}
Some of the asymptotics above have been calculated in \cite[Proof of Lemma 4.4]{bryc23asymmetric}. 

Next we prove the bound \eqref{eq:DCT bound trans}. Denote:
\begin{equation}\label{eq: Wn}
    W_n(s,t,v):= \frac{\ds(1-A_nC_n)\lb1-\frac{A_ns_n}{C_n}\rb\lb1-\frac{s_n}{t_n}\rb\sqrt{1-z_n^2}}{\ds n \abs{1-A_n\sqrt{t_n}e^{\i\theta_{z_n}}}^2  \abs{1-\frac{C_n}{\sqrt{t_n}}e^{\i\theta_{z_n}}}^2\abs{1-\frac{s_n}{C_n\sqrt{t_n}}e^{\i\theta_{z_n}}}^2}. \end{equation}
Notice that \[0<\lb|z|;q\rb_{\infty}\leq\abs{(z;q)_{\infty}}\leq \lb-|z|;q\rb_{\infty}\quad\text{for } |z|<1.\] One can observe that, as $n\rightarrow\infty$, the quotient 
\begin{multline*}
    \frac{\wt P^{(n),\rmd,\rmc}_{s,t}  \lb\wt y_s\topp n,v\rb}{W_n(s,t,v)}=  \frac{\ds\lb q,A_nB_nt_n,qA_nC_n,q\frac{A_ns_n}{C_n},B_nC_n,\frac{B_ns_n}{C_n},q\frac{s_n}{t_n};q\rb_{\infty}}{\ds \pi \qpp{A_nB_ns_n} }\\
\times   \abs{\frac{\ds\qpp{qe^{2\i\theta_{z_n}}}}{\ds\qpp{qA_n\sqrt{t_n}e^{\i \theta_{z_n}},B_n\sqrt{t_n}e^{\i\theta_{z_n}},q\frac{C_n}{\sqrt{t_n}}e^{\i \theta_{z_n}},q\frac{s_n}{C_n\sqrt {t_n}}e^{\i\theta_{z_n}}}}}^2>0
\end{multline*}
is uniformly (in $v\in[0,4n]$) bounded from above by a finite positive constant. 
Hence we only need to show that there exist constants $N,K>0$ (possibly depend on $s$ and $t$) such that for all $n>N$,
\begin{equation}\label{eq:only need}
     \abs{W_n(s,t,v)}\leq K\sqrt{v},\quad v\in[0,4n].
\end{equation}
We analyze each term in \eqref{eq: Wn}. Notice that $\sqrt{1-z_n^2}\leq\sqrt{v/n}$.  
In view of \eqref{eq:asy} and \eqref{eq:asy2}, there exist $N>0$ (possibly depends on $s$ and $t$) such that for all $n>N$, 
 we have 
\[
    \abs{1-A_nC_n}\leq\frac{\ds2\abs{\A+\C}}{\ds\sqrt{n}},\quad \abs{1-\frac{A_ns_n}{C_n}}\leq\frac{2\abs{2s+\C-\A}}{\sqrt{n}},\quad\abs{1-\frac{s_n}{t_n}}\leq\frac{2\abs{t-s}}{\sqrt{n}},
    \]
  \[
    \abs{1-A_n\sqrt{t_n}e^{\i\theta_{z_n}}}^2\geq \frac{(\A-t)^2}{2n},
\abs{1-\frac{C_n}{\sqrt{t_n}}e^{\i\theta_{z_n}}}^2\geq \frac{(\C+t)^2}{2n},\]
and
\[
\abs{1-\frac{s_n}{C_n\sqrt{t_n}}e^{\i\theta_{z_n}}}^2\geq\frac{(2s-t+\C)^2}{2n}.
\] 
By the estimates above, we conclude the proof of  bound \eqref{eq:only need} and hence \eqref{eq:DCT bound trans}. The proof is concluded.
\end{proof}

\begin{remark}\label{rem:fan} In the case $A_nC_n<1$ studied in \citep{bryc23asymmetric}, there is an intuitive and probabilistic explanation of the scaling of our interest. Namely, it is related to the notion of tangent processes in the studies of extreme values of stochastic processes. In this case, we have
\[
\Phi_d\topp n(\vv x,\vvc) = 2^{-n}\esp \lee\prodd k1d\pp{\cosh\pp{\frac{s_k}{\sqrt n}}+Y_{e^{-2s_k/\sqrt n}}\topp n}^{n_k-n_{k-1}}\ree,
\]
and $Z_n = 2^{-n}\esp(1+Y_1\topp n)^n$, 
where $Y\topp n$ is the Askey--Wilson process (mentioned in Remark \ref{rem:DEHP}) with marginal and transitional laws $\pi_t\topp n$ and $P_{s,t}\topp n$. 
What determines the limit of Laplace transform is the behavior of the process $Y\topp n$ near the upper boundary $1$ of the continuous domain $[-1,1]$, and more precisely the behavior of the transformed process
\[
\wt Y\topp n_t := 2n\pp{1-Y\topp n_{e^{2t/\sqrt n}}},\quad  t\in\R,
\]
which is also referred to as the 
(pre-limit)
 tangent process of $Y\topp n$.
It is known that as $n\rightarrow\infty$, the tangent process $\wt Y_t\topp n$ has a limit process, which has marginal and transitional laws respectively the limits of $\wt\pi_t\topp n$ and $\wt P_{s,t}\topp n$ 
with similar formulas as in Lemmas \ref{lem:marginal cont} and \ref{lem:transition}. Moreover, it is convenient to think about the tangent process jumping between the continuous part $[0,4n]$ and the atomic trajectory: for example, an integration with respect to $\wt \pi_{t_1,\dots,t_d}\topp n$ over the domain  $[0,4n]^d$ can be interpreted as an integration over the trajectories that the Markov process remains in the continuous domain (since $\wt P_{s,t}^{(n),\rmc,\rmd} = 0$ so that it cannot jump from the continuous domain to the atoms).

When $\A+\C<0$ (so that $A_nC_n>1$), this path-integration view remains to be useful. In particular, all the integrations we shall encounter later are categorized according their paths: they are either over a path that remains in the continuous domain from the very beginning, or starts in the atomic trajectory, stays there for a while, then jumps into the continuous domain and then stays there till the end. This is behind the decomposition \eqref{eq:decomposition} below.
\end{remark} 
\subsection{Limit of the Laplace transform}

\label{sec:limit Laplace}

 We need to take the 
  limit of:
\[
\varphi_{n,\vvx}\pp{\frac\vvc{\sqrt n}}  = 
\frac{\Phi_d\topp n(\vvx, \vvc)}{Z_n},
\]
as $n\to\infty$. 
Recall the integral representations of $\Phi_d\topp n(\vvx,\vvc)$ and $Z_n$ in Lemma \ref{lem:starting point}, and also that we assume $\A\ne\C$. 
We first examine the limit of~$Z_n$, and we shall need in addition $\A,\C\ne 0$. 
\begin{lemma}\label{lem:Zn}
Assume $\A+\C<0$, $\A\neq\C$, $\A,\C\neq0$ and
Assumption \ref{assump:0}. We have:
\[ 
\limn Z_n  = \calH\pp{\frac\A{\sqrt 2}, \frac\C{\sqrt 2}}  = \frac{\ds\A H(\A/ 2)-\C H(\C/ 2)}{\ds\A-\C}.
\] 
We recall that $H(x) =e^{x^2}\erfc(x)$ with $\erfc (x) = 
(2/\sqrt \pi)x\int_x^\infty e^{-t^2}\d t$.
\end{lemma}
\begin{proof} 
We first compute the limit of the integral over continuous domain $[-1,1]$:
\begin{align}
    \limn 2^{-n}\int_{[-1,1]} (1+y)^n\pi_1\topp n(\d y) &  =\limn \int_0^\infty  \inddd{u\in[0,4n]}\lb1-\frac{u}{4n}\rb^n \wt \pi_0^{(n),\rmc}(u)\d u 
    \nonumber\\
&  = \int_0^\infty e^{-u/4}\frac{\A+\C}\pi\frac{\ds\sqrt u}{(\ds\A^2+u)(\C^2+u)}\d u 
\nonumber\\
& \nonumber= \int_0^\infty \pp{\frac1{\A^2+u}-\frac1{\C^2+u}}\frac{\ds\sqrt u}{\ds\pi(\C-\A)}e^{-u/4}\d u  \\&= \frac{\ds|\C| H(|\C|/2)-|\A|H(|\A|/2)}{\ds\C-\A}.\label{eq:continuous part} 
\end{align} 
The first step uses the transformation  $u=2n(1-y)$. The second step combines $\limn(1-u/(4n))^n= e^{-u/4}$,
\[ \lim_{n\rightarrow\infty}\wt \pi_0^{(n),\rmc}(u)=\frac{\ds\A+\C}{\ds\pi}\frac{\ds\sqrt u}{\ds(\A^2+u)(\C^2+u)}\]
(which follows from taking $t=0$ in \eqref{eq:lim pi c} in Lemma \ref{lem:marginal cont}) and the dominated convergence theorem 
using the bound \eqref{eq:DCT bound mar} in Lemma \ref{lem:marginal cont}. The last step uses the identity \citep[equation (4.38)]{bryc23asymmetric}:  
\[\int_0^\infty \frac{\sqrt u}{\alpha^2+u}e^{-u/4}du = \pi\pp{\frac2{\sqrt \pi} - |\alpha| H(|\alpha|/2)},\quad\text{ for any }\alpha\in\R.\]

Next we consider the contribution  from atoms. Recall the discussion after \eqref{eq:t_n}. There are three cases (we are with $t=0$ and hence $t_n = 1$):
\begin{enumerate}[(i)]
\item Case $0<\A<-\C$: for sufficiently large $n$ there is an atom generated by $C_n$ with value 
\[
y_1\topp{C_n}:=\frac{1}{2}\lb C_n+\frac{1}{C_n}\rb,\hspace{.15cm}\text{hence}\hspace{.15cm} 2^{-n}\lb1+y_1\topp{C_n}\rb^n=\lb1+\frac{1}{4C_n}(C_n-1)^2\rb^n\to e^{\C^2/4}. 
\]
This atom has mass
\begin{equation}\label{eq:mass case 1}
    \pi^{(n)}_1\lb \ccbb{y_1\topp{C_n}}\rb=\frac{\ds \qpp{\frac{1}{C_n^2},A_nD_n, B_nD_n, A_nB_n}}{\ds \qpp{\frac{D_n}{C_n},\frac{A_n}{C_n},\frac{B_n}{C_n},A_nB_nC_nD_n}}\sim\frac{\ds1-\frac{1}{C_n^2}}{\ds1-\frac{A_n}{C_n}} \rightarrow\frac{\ds2\C}{\ds\C-\A}.\end{equation}
\item Case $\A<-\C<0$: for sufficiently large $n$ there is an atom generated by $A_n$ with value
\[
y_1\topp{A_n}=\frac{1}{2}\lb A_n+\frac{1}{A_n}\rb,
\]
and we have
\[
 2^{-n}\lb1+y_1\topp{A_n}\rb^n=\lb1+\frac{1}{4A_n}(A_n-1)^2\rb^n\to e^{\A^2/4}.
 \] 
This atom has mass 
\begin{equation}\label{eq:mass case 2}
    \pi^{(n)}_1\lb \ccbb{y_1\topp{A_n}}\rb=\frac{\ds \qpp{\frac{1}{A_n^2},C_nD_n, B_nD_n, C_nB_n}}{\ds \qpp{\frac{D_n}{A_n},\frac{C_n}{A_n},\frac{B_n}{A_n},A_nB_nC_nD_n}}\sim\frac{\ds1-\frac{1}{A_n^2}}{\ds1-\frac{C_n}{A_n}} \rightarrow\frac{\ds-2\A}{\ds\C-\A}.
\end{equation} 
\item Case $\A<0<-\C$: both atoms above exist and with the same mass value \eqref{eq:mass case 1} and \eqref{eq:mass case 2} respectively.  
\end{enumerate} 
Therefore, it follows that
\begin{equation}\label{eq:contribution from atoms}
    \limn 2^{-n}\int_{\R\setminus[-1,1]} (1+y)^n\pi_1\topp n(\d y)=  
\begin{cases}
\displaystyle \frac{2\C}{\C-\A}e^{\C^2/4},& \mbox{ if } 0<\A<-\C,\\\\
\displaystyle \frac{-2\A}{\C-\A}e^{\A^2/4}, & \mbox{ if } \A<-\C<0,\\\\
\displaystyle \frac{2\C e^{\C^2/4}-2\A e^{\A^2/4}}{\C-\A},& \mbox{ if } \A<0<-\C.
\end{cases}
\end{equation}  
Combining \eqref{eq:continuous part} and \eqref{eq:contribution from atoms}, we have:
\[ 
\limn Z_n=\limn 2^{-n}\int_{\R}(1+y)^n \pi_1\topp n(\d y)  = \frac{\ds\A H(\A/ 2)-\C H(\C/ 2)}{\ds\A-\C}.
\] 
 We conclude the proof of the lemma.
\end{proof}
\begin{proposition}\label{prop:limit Laplace}
Assume $\A+\C<0$, $\A\neq\C$, $\A,\C\neq0$ and
Assumption \ref{assump:0}. 
For any $\vvc=(c_1,\dots,c_d)\in\R_+^d$  satisfying $-c_d<\A<-\C$  and any $\vv x = (x_1,\dots,x_d)$ satisfying $0=x_0<x_1<\cdots<x_d= 1$,
we have:
\[ 
\limn\varphi_{n,\vvx}\pp{\frac\vvc{\sqrt n}} = \limn
\frac{\Phi_d\topp n(\vvx, \vvc)}{Z_n} = \PhiB(\vvx,\vvc) \frac{\summ \ell0d \Phi_{d,\ell}(\vv x,\vvc)}{\calH\lb\A/\sqrt 2,\C/\sqrt 2\rb},
\] 
where
\[ 
\PhiB(\vvx,\vvc) = \exp\pp{\frac14\summ k1d s_k^2 (x_k-x_{k-1})}=\esp\lee\exp\pp{-\frac1{\sqrt 2}\summ k1d c_k\BB_{x_k}}\ree
\] 
is the Laplace transform of finite-dimensional distributions of a Brownian motion, 
and the quantities $\Phi_{d,\ell}$ for $\ell=0,\dots,d$ are given by \eqref{eq:wt dd}, \eqref{eq:wt d0}  and \eqref{eq:wt dl}   below. 
We recall that $s_k = c_k+\cdots+c_d$ for $k=1,\dots, d$. 
\end{proposition}
\begin{proof}
We have computed the limit of $Z_n$ in Lemma \ref{lem:Zn}. It remains to examine the limit of $\Phi_d\topp n(\vvx,\vvc)$. 

Denote  $\vvu=(u_1,\dots,u_d)$.
We write:
\begin{align*}
\Phi_{d}\topp n(\vvx,\vvc) &=
2^{-n}\int_\Rd\prodd k1d\pp{\cosh\pp{\frac{s_k}{\sqrt n}}+y_k}^{n_k-n_{k-1}} \pi\topp n_{e^{-2s_1/\sqrt n},\dots,e^{-2s_d/\sqrt n}}(\d y_1,\dots,\d y_d)
\\
& =  \int_{\R^d}  G_{\vvx,\vvc,n}(\vvu)\wt\pi\topp n_{-s_1,\dots,-s_d}(\d\vv u),
\end{align*}
where $n_k = \floor{n x_k}$ for $k=0,\dots,d$ and
\[
G_{\vvx,\vvc,n}(\vvu) = \prodd k1d\pp{1+\sinh^2\pp{\frac{s_k}{2\sqrt n}}-\frac{u_k}{4n}}^{n_k-n_{k-1}}.
\]  
 
We next write $\Phi_d\topp n$ as a sum of multiple integrals by decomposing the domain of the integral, $\Rd$:
\equh\label{eq:decomposition}
\Phi_d\topp n(\vvx,\vvc) =\summ \ell0d\Phi_{d,\ell}\topp n(\vvx,\vvc),
\eque
with
\begin{align*}
\Phi_{d,0}\topp n(\vvx,\vvc)& := \int_{\R^d}\prodd k1d G_{\vvx,\vvc,n}(\vvu)\inddd{u_1\geq0}\wt\pi\topp n_{-s_1,\dots,-s_d}(\d\vv u),\\
\Phi_{d,\ell}\topp n(\vvx,\vvc)& := \int_{\R^d}\prodd k1d G_{\vvx,\vvc,n}(\vvu)\inddd{u_1,\dots,u_\ell<0,u_{\ell+1}\geq0}\wt\pi\topp n_{-s_1,\dots,-s_d}(\d\vv u),\quad \ell=1,\dots,d-1,\\
\Phi_{d,d}\topp n(\vvx,\vvc)& := \int_{\R^d}\prodd k1d G_{\vvx,\vvc,n}(\vvu)\inddd{u_1,\dots,u_d<0}\wt\pi\topp n_{-s_1,\dots,-s_d}(\d\vv u).
\end{align*}
Recall from Lemma \ref{lem:marginal atom} that for each $k=1,\dots,d$, when $n$ gets large enough,  $\wt\pi_{-s_k}\topp n$ has a single atom at $\wt y\topp n_{-s_k}<0$. Moreover, by Lemma \ref{lem:transition}, for large enough $n$,  $\wt P_{-s_{k-1},-s_k}^{(n)}$ does not have atoms. 
Denote $\vvu_{\ell:d}=(u_{\ell},\dots,u_d)$ for $\ell=1,\dots,d$, where $\vvu_{1:d}=\vvu$.
We have the explicit expressions:
\begin{align}
      \label{eq:alt form d0}  \Phi_{d,0}\topp n(\vvx,\vvc)   &= \int_{\R_+^d}G_{\vvx,\vvc,n}(\vvu)\wt \pi_{-s_1}^{(n),\rmc}(u_1)\prodd k2d\wt P_{-s_{k-1},-s_k}^{(n),\rmc,\rmc}(u_{k-1},u_k)\d\vv u,\\
      \label{eq:alt form dl}\Phi_{d,\ell}\topp n(\vvx,\vvc) &=\wt \pi_{-s_1}^{(n),\rmd}\lb y_{-s_1}\topp n\rb\prodd j2\ell \wt P^{(n),\rmd,\rmd}_{-s_{j-1},s_j}\lb\wt y_{-s_{j-1}}\topp n,\wt y_{-s_j}\topp n\rb\int_{\R_+^{d-\ell}}G_{\vvx,\vvc,n}\pp{\wt y\topp n_{-s_1},\dots,\wt y\topp n_{-s_\ell},\vvu_{\ell+1:d}} \\
    \nonumber  &\quad\times\wt P_{-s_\ell,-s_{\ell+1}}^{(n),\rmd,\rmc}\lb\wt y\topp n_{-s_{\ell}},u_{\ell+1}\rb 
\prodd k{\ell+2} d \wt P_{-s_{k-1},-s_k}^{(n),\rmc,\rmc}\lb u_{k-1},u_k\rb\d \vvu_{\ell+1:d},\hspace{.1cm} \ell=1,\dots,d-1,
\end{align} 
and
\[
\Phi_{d,d}\topp n(\vvx,\vvc) = G_{\vvx,\vvc,n}\pp{\wt y\topp n_{-s_1},\dots,\wt y\topp n_{-s_d}} \wt \pi_{-s_1}^{(n),\rmd}\lb\wt y\topp n_{-s_1}\rb\prodd k2d \wt P_{-s_{k-1},-s_k}^{(n),\rmd,\rmd}\lb\wt y_{-s_{k-1}}\topp n,\wt y_{-s_k}\topp n\rb.
\]
A helpful way of thinking of decomposition \eqref{eq:decomposition} and the explicit forms above was provided in Remark \ref{rem:fan}.

We first analyze the limit of $\Phi_{d,d}\topp n(\vvx,\vvc)$ as $n\to\infty$. 
We have (see \cite[Proof of Lemma 4.7]{bryc23asymmetric}):
\[ 
\limn G_{\vvx,\vvc,n}(\vvu) = G_{\vvx,\vvc}(\vvu), \quad u\in[0,\infty)^d,
\] 
with
\[ 
 G_{\vvx,\vvc}(\vvu) := \exp\pp{\frac14\summ k1d (s_k^2-u_k)(x_k-x_{k-1})} = \PhiB(\vvx,\vv c)  \exp\pp{-\frac14\summ k1d u_k\Delta x_k },
\] 
and 
\[
\Delta x_k = x_k-x_{k-1},\quad k=1,\dots,d.
\]
We also note a slightly stronger convergence:
\begin{equation}\label{eq:lim of G strong}
\limn G_{\vvx,\vvc,n}(\vvu_n) = G_{\vvx,\vvc}(\vvu),\quad \text{for }\vvu_n\to\vvu\in[0,\infty)^d \text{ as }n\to\infty. 
\end{equation}
Moreover, we have \cite[equation (4.32)]{bryc23asymmetric}: for some $K>0$ independent of $n$,
\begin{equation}\label{eq:bound on Gn}
    G_{\vvx,\vvc,n}(\vvu) \le K \prodd k1d\exp\pp{-\frac{\ds n_k-n_{k-1}}{\ds 4n}u_k},\quad \vv u \in[0,4n]^d.
\end{equation}
In view of the limit of $\wt y\topp n_{t}$ and  $\wt \pi_{t}^{(n),\rmd}(\wt y\topp n_{t})$ in Lemma \ref{lem:marginal atom}, the limit  
 of $\wt P_{s,t}^{(n),\rmd,\rmd}(\wt y_{s}\topp n,\wt y_{t}\topp n)$ in Lemma \ref{lem:transition}, and the limit of  $G_{\vvx,\vvc,n}(\vvu)$ in \eqref{eq:lim of G strong}, we have: 
\[\limn\Phi_{d,d}\topp n(\vvx,\vvc)=G_{\vvx,\vvc}\pp{y^{*,\C}_{-s_1},\dots,y^{*,\C}_{-s_d}}\frac{2\C-2s_d}{\C-2s_d-\A}   = \PhiB(\vvx,\vvc)
\Phi_{d,d}
(\vvx,\vvc),\]
where we recall $y^{*,\C}_t=-(t+\C)^2$, and denote:
\equh\label{eq:wt dd}
\Phi_{d,d}(\vvx,\vvc) :=  \exp\pp{\frac14\summ k1d \Delta x_k(s_k-\C)^2}  \frac{2\C-2s_d}{\C-2s_d-\A}.
\eque

To analyze the limits of $\Phi_{d,0}\topp n(\vvx,\vvc)$ and $\Phi_{d,\ell}\topp n(\vvx,\vvc)$, $\ell=1,\dots,d-1$, we need the following fact.
\begin{lemma}[Lemma 4.8 in \citep{bryc23asymmetric}, see also Theorem 5.5 in \citep{billingsley99convergence}]
\label{lem:cite}
    Suppose a sequence of probability measures $\mu\topp n$ on  $\R_+^r$  converges weakly to a probability measure $\mu$. Let  $G_{n}$ be a   sequence of uniformly bounded measurable functions, $G_n:\R_+^r\to [-K,K]$ for some $K>0$,   such that for all $\vvu_n\to \vv u\in\R_+^r$ we have
\[\limn G_{n}(\vv u_n) = G(\vv u). 
\]
Then we have:
\[ 
\limn \int G_n(\vv u)\mu\topp  n(\d \vv u)= \int G(\vv u)\mu(\d \vv u).
\] 
\end{lemma}

We next analyze the limit as $n\to\infty$ of $\Phi_{d,0}\topp n(\vvx,\vvc)$ given by \eqref{eq:alt form d0}. We write:
\[ 
    \begin{split}
        \Phi_{d,0}\topp n(\vvx,\vvc) & =  \int_{\R_+}\wt \pi_{-s_1}^{(n),\rmc}(u_1)\d u_1\int_{\R_+^{d-1}}G_{\vvx,\vvc,n}(\vvu)\prodd k2d\wt P_{-s_{k-1},-s_k}^{(n),\rmc,\rmc}(u_{k-1},u_k)\d
        \vvu_{2:d} \\
        &= \int_{\R_+}\inddd{u_1\in[0,4n]}\wt \pi_{-s_1}^{(n),\rmc}(u_1)\mathsf{G}_{d,0}\topp n(u_1)\d u_1,
    \end{split}
\] 
where 
\begin{equation}\label{eq:def of Gn0}
    \mathsf{G}_{d,0}\topp n(u_1):=\int_{\R_+^{d-1}}G_{\vvx,\vvc,n}(\vvu)\mu_{d,0,u_1}\topp n\lb\d\vvu_{2:d}\rb,
\end{equation}
with
\[\mu_{d,0,u_1}\topp n\lb\d\vvu_{2:d}\rb:=\prodd k2d\wt P_{-s_{k-1},-s_k}^{(n),\rmc,\rmc}(u_{k-1},u_k)\d \vvu_{2:d},\quad u_1\in[0,4n]\]
is a probability measure on $\R_+^{d-1}$ for any fixed $u_1$. Using Scheffé's lemma and the limit of $\wt P_{s,t}^{(n),\rmc,\rmc}(u,v)$ from Lemma \ref{lem:transition}, as $n\to\infty$, probability measures  $\mu_{d,0,u_1}\topp n$ weakly converge to $\mu_{d,0,u_1}$ given by:
\[
\mu_{d,0,u_1}\lb\d\vvu_{2:d}\rb=\frac{\ds (\A+s_1)^2+u_1}{\ds (\A+s_d)^2+u_d}\prodd k2d \p_{c_{k-1}}(u_{k-1},u_k) \d\vvu_{2:d},\quad u_1\in[0,4n],
\]
where the function $\p_t(u,v)$ is defined in \eqref{eq:def of p}.
Using Lemma \ref{lem:cite} for $r=d-1$, in view of \eqref{eq:lim of G strong}, we have:
\[
    \limn\mathsf{G}_{d,0}\topp n(u_1)=\mathsf{G}_{d,0}(u_1):=\int_{\R_+^{d-1}}G_{\vvx,\vvc}(\vvu)\mu_{d,0,u_1} \lb\d\vvu_{2:d}\rb,\quad u_1\in[0,4n].
\]
In view of \eqref{eq:def of Gn0}, using the bound \eqref{eq:bound on Gn},  we have:
\begin{equation}
    \label{eq:bound Gd0}
 \mathsf{G}_{d,0}\topp n(u_1)\leq K\exp\pp{-\frac{n_1}{4n}u_1},\quad   u_1 \in[0,4n].\end{equation}
By the pointwise limit of $\wt\pi_t^{(n),\rmc}(u)$ in Lemma \ref{lem:marginal cont}, \eqref{eq:bound Gd0}, and the uniform bound of $\abs{\wt\pi_t^{(n),\rmc}(u)}$ in Lemma \ref{lem:marginal cont}, one can use the dominated convergence theorem to  show 
\begin{align*}
\Phi_{d,0}\topp n(\vvx,\vvc) &=\int_{\R_+}\inddd{u_1\in[0,4n]}\wt \pi_{-s_1}^{(n),\rmc}(u_1)\mathsf{G}_{d,0}\topp n(u_1)\d u_1\\
& \to \int_{\R_+}\frac{\A+\C}\pi \frac{\ds \sqrt{u_1}}{\ds ((\A+s_1)^2+u_1)((\C-s_1)^2+u_1)}\mathsf{G}_{d,0}(u_1)\d u_1=\PhiB(\vvx,\vvc)
\Phi_{d,0}
(\vvx,\vvc),
\end{align*} 
as $n\to\infty$,
where
\begin{multline}\label{eq:wt d0}
    \Phi_{d,0}(\vvx,\vvc) \\
    :=   \frac{\ds \A+\C}{\ds \pi}  \int_{\R_+^d}\d\vvu\exp\lb-\frac14\summ k1d u_k\Delta x_k \rb\frac{\ds \sqrt{u_1}}{\ds (\C-s_1)^2+u_1}\frac{\ds 1}{\ds (\A+s_d)^2+u_d}\prodd k2d \p_{c_{k-1}}(u_{k-1},u_k).
\end{multline}

It remains to analyze the limits as $n\to\infty$ of $\Phi_{d,\ell}\topp n(\vvx,\vvc)$, $\ell=1,\dots,d-1$ given by \eqref{eq:alt form dl}.
This procedure is similar to the previous case.
We write, first assuming in addition $\ell\le d-2$,
\begin{align}
        \Phi_{d,\ell}\topp n(\vvx,\vvc) &=\wt \pi_{-s_1}^{(n),\rmd}\lb y_{-s_1}\topp n\rb\prodd j2\ell \wt P^{(n),\rmd,\rmd}_{-s_{j-1},s_j}\lb\wt y_{-s_{j-1}}\topp n,\wt y_{-s_j}\topp n\rb\int_{\R_+} \wt P_{-s_\ell,-s_{\ell+1}}^{(n),\rmd,\rmc}\lb\wt y\topp n_{-s_{\ell}},u_{\ell+1}\rb\d u_{\ell+1}\nonumber\\ 
   &\quad\quad\times \int_{\R_+^{d-\ell-1}}  G_{\vvx,\vvc,n}\pp{\wt y\topp n_{-s_1},\dots,\wt y\topp n_{-s_\ell},\vvu_{\ell+1:d}}
\prodd k{\ell+2} d \wt P_{-s_{k-1},-s_k}^{(n),\rmc,\rmc}\lb u_{k-1},u_k\rb\d \vvu_{\ell+2:d} \nonumber\\
&=\wt \pi_{-s_1}^{(n),\rmd}\lb y_{-s_1}\topp n\rb\prodd j2\ell \wt P^{(n),\rmd,\rmd}_{-s_{j-1},s_j}\lb\wt y_{-s_{j-1}}\topp n,\wt y_{-s_j}\topp n\rb  \label{eq:analyze l}\\
 &\quad\quad\times \int_{\R_+}\inddd{u_{\ell+1}\in[0,4n]}\mathsf{G}_{d,\ell}\topp n(u_{\ell+1}) \wt P_{-s_\ell,-s_{\ell+1}}^{(n),\rmd,\rmc}\lb\wt y\topp n_{-s_{\ell}},u_{\ell+1}\rb\d u_{\ell+1},\nonumber
\end{align} 
where 
\begin{equation}\label{eq:def of Gnl}
    \mathsf{G}_{d,\ell}\topp n(u_{\ell+1}):=\int_{\R_+^{d-\ell-1}}G_{\vvx,\vvc,n}\pp{\wt y\topp n_{-s_1},\dots,\wt y\topp n_{-s_\ell},\vvu_{\ell+1:d}} \mu_{d,\ell,u_{\ell+1}}\topp n\lb\d \vvu_{\ell+2:d}\rb, 
\end{equation}
and
\[\mu_{d,\ell,u_{\ell+1}}\topp n\lb\d\vvu_{\ell+2:d}\rb=\prodd k{\ell+2} d \wt P_{-s_{k-1},-s_k}^{(n),\rmc,\rmc}\lb u_{k-1},u_k\rb\d \vvu_{\ell+2:d},\quad u_{\ell+1}\in[0,4n]\]
is a probability measure on $\R_+^{d-\ell-1}$ for all $u_{\ell+1}$. Using the Scheffé's lemma and the limit of $\wt P_{s,t}^{(n),\rmc,\rmc}(u,v)$ from Lemma \ref{lem:transition}, as $n\to\infty$, probability measures  $\mu_{d,\ell,u_{\ell+1}}\topp n$ weakly converge to $\mu_{d,\ell,u_{\ell+1}}$ given by:
\[\mu_{d,\ell,u_{\ell+1}}\lb\d \vvu_{\ell+2:d}\rb=\frac{\ds (\A+s_{\ell+1})^2+u_{\ell+1}}{\ds (\A+s_d)^2+u_d}\prodd k{\ell+2}d \p_{c_{k-1}}(u_{k-1},u_k) \d \vvu_{\ell+2:d},\quad u_1\in[0,4n].\]
Using Lemma \ref{lem:cite} for $r=d-\ell-1$, in view of \eqref{eq:lim of G strong}, we have:
\[
    \limn\mathsf{G}_{d,\ell}\topp n(u_{\ell+1})=\mathsf{G}_{d,\ell} (u_{\ell+1}),\quad u_{\ell+1}\in[0,4n],
\]
where 
\[
\mathsf{G}_{d,\ell} (u_{\ell+1}):=\int_{\R_+^{d-\ell-1}}G_{\vvx,\vvc}\pp{\wt y\topp n_{-s_1},\dots,\wt y\topp n_{-s_\ell},\vvu_{\ell+1:d}}\mu_{d,\ell,u_{\ell+1}} \lb\d \vvu_{\ell+2:d}\rb.
\]
In view of \eqref{eq:def of Gnl}, using the bound \eqref{eq:bound on Gn},  we have:
\begin{equation}
    \label{eq:bound Gdl}
 \mathsf{G}_{d,\ell}\topp n(u_{\ell+1})\leq K\exp\pp{-\frac{\ds n_{\ell+1}-n_{\ell}}{\ds 4n}u_{\ell+1}},\quad   u_{\ell+1} \in[0,4n].\end{equation}
By \eqref{eq:bound Gdl} and the uniform bound of $\sabs{\wt P_{s,t}^{(n),\rmd,\rmc}(\wt y\topp n_{s},v)}$ in Lemma \ref{lem:transition}, one can use the dominated convergence theorem to evaluate the limit of  integral in \eqref{eq:analyze l}. By the pointwise limits in Lemma \ref{lem:marginal atom} and Lemma \ref{lem:transition}:

\begin{align*} 
        \limn &\Phi_{d,\ell}\topp n(\vvx,\vvc)
     \\
       &=\limn\wt \pi_{-s_1}^{(n),\rmd}\lb y_{-s_1}\topp n\rb\prodd j2\ell\limn \wt P^{(n),\rmd,\rmd}_{-s_{j-1},s_j}\lb\wt y_{-s_{j-1}}\topp n,\wt y_{-s_j}\topp n\rb  \\
 &\quad\quad\times \limn\int_{\R_+}\inddd{u_{\ell+1}\in[0,4n]}\mathsf{G}_{d,\ell}\topp n(u_{\ell+1}) \wt P_{-s_\ell,-s_{\ell+1}}^{(n),\rmd,\rmc}\lb\wt y\topp n_{-s_{\ell}},u_{\ell+1}\rb\d u_{\ell+1}\\
 &= \frac{\A+\C}\pi\frac{\ds 2(\C-s_1)}{\ds -2s_1+\C-\A}\prodd j2\ell\frac{\ds (\C-s_j)(-2s_{j-1}+\C-\A )}{\ds (\C-s_{j-1})(-2s_j+\C-\A )}\\
 &\quad\quad\times  \int_{\R_+^{d-\ell}}G_{\vvx,\vvc}\pp{y^{*,\C}_{-s_{1}},\dots,y^{*,\C}_{-s_\ell},\vvu_{\ell+1:d}}
\frac{\ds (\A+s_{\ell+1})^2+u_{\ell+1}}{\ds (\A+s_d)^2+u_d}\prodd k{\ell+2}d \p_{c_{k-1}}(u_{k-1},u_k)\\
&\quad\quad\times\frac{\ds \sqrt{u_{\ell+1}}}{\ds ((\A+s_{\ell+1})^2+u_{\ell+1})((\C-s_{\ell+1})^2+u_{\ell+1})}  \frac{\ds (\A-\C+2s_{\ell})(2s_{\ell}-2s_{\ell+1})}{\ds (-2s_{\ell}+s_{\ell+1}+\C)^2+u_{\ell+1}}\d \vvu_{\ell+1} \\
&=\PhiB(\vvx,\vvc)\Phi_{d,\ell}(\vvx,\vvc), 
\end{align*}
where 
\begin{multline}\label{eq:wt dl}
    \Phi_{d,\ell}(\vvx,\vvc) :=      \frac{\A+\C}\pi\int_{\R_+^{d-\ell}}
    e^{-\frac14\summ k{\ell+1}d\Delta x_ku_k+\frac14\summ k1\ell\Delta x_k(s_k-\C)^2} 
   \prodd k{\ell+2}d \p_{c_{k-1}}(u_{k-1},u_k)\\
   \times  
 \sqrt{u_{\ell+1}}\frac1{(\A+s_d)^2+u_d}  
   \frac{\ds 4(\C-s_\ell)(s_{\ell+1}-s_\ell)}{\ds ((\C-s_{\ell+1})^2+u_{\ell+1})((\C-2s_\ell+s_{\ell+1})^2+u_{\ell+1})}\d\vvu_{\ell+1:d}.
\end{multline}
It remains to prove the case $\ell = d-1$. In this case we have in place of \eqref{eq:def of Gnl}
\[
\mathsf G_{d,d-1}\topp n:=G_{\vvx,\vvc,n}\pp{\wt y_{-s_1}\topp n,\dots,\wt y_{-s_{d-1}}\topp n,u_d}.
\]
The rest of the analysis is similar. We conclude the proof.
\end{proof}
\subsection{\texorpdfstring{Proof of the main theorem for $\A+\C<0$,  $\A\neq\C$ and $\A,\C\neq0$}{}} \label{subsec:proof a neq c} 
We recall that the tightness part of Theorem \ref{thm:1} was established right before Section \ref{sec:tangent}. It remains to be shown that, as $n\to\infty$,
\equh\label{eq:fdd again}
\frac1{\sqrt n}\lb h_n(x)\rb_{x\in[0,1]}\fddto \frac1{\sqrt 2}\lb\BB_x + \eta\topp{\A/\sqrt 2,\C/\sqrt 2}_x\rb_{x\in[0,1]}.
\eque
We have computed the limit Laplace transform of finite-dimensional distributions of the left-hand side above in Proposition \ref{prop:limit Laplace} as
\begin{equation}
    \label{eq:coafwe}
\limn\varphi_{n,\vvx}^{(A_n,B_n,C_n,D_n,q)}\pp{\frac\vvc{\sqrt n}}  = 
\PhiB(\vvx,\vvc) \frac{\summ \ell0d \Phi_{d,\ell}(\vv x,\vvc)}{\calH\lb\A/\sqrt 2,\C/\sqrt 2\rb}.
\end{equation}
The desired \eqref{eq:fdd again} will follow right away from the convergence of corresponding Laplace transforms: that is, \eqref{eq:coafwe} and the following proposition.
\begin{proposition}\label{prop:duality}
With $c_1,\dots,c_d>0, -c_d<\A<-\C$, $0=x_0<x_1<\cdots<x_d=1$, we have
\equh\label{eq:conjecture}
\frac1{\calH(\A/\sqrt 2,\C/\sqrt 2)}\summ \ell0d\Phi_{d,\ell}(\vvx,\vvc) =\mathbb{E}\lee\exp\pp{-\frac{1}{\sqrt{2}}\summ k1d c_k \eta\topp{\A/\sqrt 2,\C/\sqrt 2}_x}\ree.
\eque
\end{proposition}
\begin{remark}\label{rem:duality}
A corresponding identity to \eqref{eq:conjecture} showed up in closely related investigations of open ASEP in \citep{bryc19limit,wang24askey,bryc23asymmetric,bryc23markov}, where each side can be interpreted as a Laplace transform of certain stochastic process (possibly with respect to an infinite measure). Here, the right-hand side is already a Laplace transform while whether the left-hand side has such an interpretation is much less clear. Nevertheless, it is clear that the right-hand side involves integrals with respect to transitional kernel $\q$ and $\q^{(b),*}$, while the left-hand side with respect to $\p$. The kernels $\p$ and $\q$ are in a dual relation in an appropriate sense as in all other examples: see \citep{kuznetsov24dual} for a general framework. We therefore refer to \eqref{eq:conjecture} as a duality formula. 

We actually managed to rewrite $\Phi_{d,\ell}$ as an integral involving the kernel $\q$ too, but we still could not prove \eqref{eq:conjecture} directly. See Appendix \ref{sec:duality} and Proposition \ref{prop:duality1} for more details. Instead, our proof of Proposition \ref{prop:duality} makes use of a recent key development by \citet{bryc24two}. 
\end{remark}

\begin{proof}[Proof of Proposition \ref{prop:duality}]
We only prove the case under the additional assumptions that $\A\ne \C$ and $\A,\C\ne 0$. These conditions can then be dropped by a continuity argument and we omit the details.

The idea is to first show that \eqref{eq:fdd again} and \eqref{eq:coafwe} imply \eqref{eq:conjecture}, {\em for a specific choice of parameters of $A_n, B_n, C_n, D_n,q$}. The key observation is that the identity \eqref{eq:conjecture} depends only on $\A,\C$ (see  \eqref{eq:wt dd}, \eqref{eq:wt d0}  and \eqref{eq:wt dl}), and that with $\A,\C$ fixed we still have some freedom on the choice of $A_n, B_n, C_n, D_n, q$ (in order to prove \eqref{eq:fdd again}).

Indeed, the convergence \eqref{eq:fdd again} has been shown in \citep[Theorem 3.1]{bryc24two} for open TASEP. Specifically, for any $\A,\C\in\R$, 
the convergence \eqref{eq:fdd again} holds with 
\[ 
    A_n=e^{-\A/\sqrt{n}},\quad B_n=0,\quad C_n=e^{-\C/\sqrt{n}},\quad D_n=0, \quad q=0.
    \] 
(In fact, they showed this convergence in the Skorokhod’s space $D([0,1])$.)
It is immediate that this choice of parameters satisfies Assumption \ref{assump:0}. 
Now, restricting further to  $\A+\C<0, \A\ne \C, \A,\C\ne 0$, 
and recalling our assumptions that $c_1,\dots,c_d>0$ and $-c_d<\A$,
we also have \eqref{eq:coafwe} following Proposition \ref{prop:limit Laplace}.

It remains to show that convergence in distribution here implies convergence of the corresponding Laplace transforms. For this purpose it remains to check uniform integrability. 
It follows from \citep[Theorem 3.1]{bryc24two}  that 
\begin{equation}\label{eq:seqce}
\exp\pp{-\frac{1}{\sqrt{n}}\summ k1d c_k h_n(x_k) } \weakto\exp\pp{-\frac{1}{\sqrt{2}}\summ k1d c_k \lb\BB_x + \eta\topp{\A/\sqrt 2,\C/\sqrt 2}_x\rb}.
\end{equation}
In view of \eqref{eq:coafwe} but changing $c_k$ to $pc_k$ for $k=1,\dots,d$ and for some $p>1$, 
 we justify the uniform integrability of sequence of left-hand side of \eqref{eq:seqce} indexed by $n\in\N$ 
 (e.g.~\cite[Theorem 4.6.2]{durrett10probability}). 
This change of variables is valid because $pc_1, \dots, pc_d > 0$ and $-pc_d < \A$, which follow from our assumptions that $c_1, \dots, c_d > 0$, $-c_d < \A$, and $p > 1.$ 
 Therefore,  using \cite[Theorem 3.5]{billingsley99convergence} we have:
\[
        \limn\varphi_{n,\vvx}^{(A_n,B_n,C_n,D_n,q)}\pp{\frac\vvc{\sqrt n}}  = 
        \mathbb{E}\lee\exp\pp{-\frac{1}{\sqrt{2}}\summ k1d c_k \lb\BB_x + \eta\topp{\A/\sqrt 2,\C/\sqrt 2}_x\rb}\ree.
\]
Comparing the above with \eqref{eq:coafwe}, we have proved the desired \eqref{eq:conjecture}.
\end{proof}

\subsection{Proof of the main theorem for the remaining cases} \label{subsec:proof for a=c}
In this subsection we will prove Theorem \ref{thm:1} for the remaining special cases: $\A+\C<0$, and either $\A=\C$ or 
of one $\A,\C$ equals zero,
 using 
 a
 continuity argument. 

The following result is known as the `stochastic sandwiching' of stationary measures of open ASEP.
\begin{lemma}\label{lem:sandwiching}
Consider the open ASEP on the lattice $\{1,\dots,n\}$ with parameters $(A',B',C',D',q)$ and $(A'',B'',C'',D'',q)$, where $q\in[0,1)$ and
\begin{equation}
    \label{eq:agewr}
0\leq A'\leq A'',\quad  B'=B''\in(-1,0],\quad   C'\geq C''\geq0,\quad  D'=D''\in(-1,0].
\end{equation}
The stationary measures are denoted respectively as $\mu'$ and $\mu''$, both with occupation variables $(\tau_{n,1},\dots,\tau_{n,n})$. 
Then for any $0<f_1,\dots,f_n\leq1$, we have 
$\mathbb{E}_{\mu'_n}[\prod_{i=1}^nf_i^{\tau_{n,i}}]\geq\mathbb{E}_{\mu''_n}[\prod_{i=1}^nf_i^{\tau_{n,i}}].$
\end{lemma}
\begin{proof}
The open ASEP jump rates $\alpha,\beta,\gamma,\delta$ can be expressed in terms of parameters $A,B,C,D$ through the following inverse transformation of \eqref{eq:defining ABCD} (see \cite[equation (2.4)]{bryc17asymmetric}):
    \[\alpha=\frac{1-q}{(1+C)(1+D)},\quad \beta=\frac{1-q}{(1+A)(1+B)},\quad \gamma=\frac{-(1-q)CD}{(1+C)(1+D)},\quad \delta=\frac{-(1-q)AB}{(1+A)(1+B)}.\]
It then follows from \eqref{eq:agewr} that:
\[0<\alpha' \leq\alpha'',\quad\beta' \geq\beta''>0,\quad\gamma' \geq\gamma''\geq0,\quad0\leq\delta' \leq\delta''.\]
By the stochastic sandwiching result \cite[Lemma 2.1]{gantert23mixing}, there exists a coupling of $\mu'_n$ and $\mu_n''$ such that almost surely $\tau'_{n,i}\leq\tau_{n,i}''$ for $i=1,\dots,n$, where $\tau'_{n,i}$ and $\tau''_{n,i}$ respectively denote the occupation variables of $\mu_n'$ and $\mu_n''$. The inequality of the Laplace transform directly follows from this coupling. We conclude the proof. 
\end{proof}

\begin{proof}[Proof of Theorem \ref{thm:1}]
We have shown in Section \ref{subsec:proof a neq c} that for $\A\neq\C$ and $\A,\C\neq0$ 
the desired convergence. That is, under Assumption \ref{assump:0} and 
assuming $\vvc=(c_1,\dots,c_d)\in\R_+^d$ satisfies $-c_d<\A<-\C$,
the Laplace transform
\[ \varphi_{n,\vvx}^{(A_n,B_n,C_n,D_n,q)}\pp{\frac\vvc{\sqrt n}}  =  
\esp
\lee\exp\pp{-\frac{1}{\sqrt{n}}\summ k1d c_k h_n(x_k) }\ree\]
converges to the Laplace transform $\varphi_\vvx^{(\A,\C)}(\vvc)$ of the limit process (defined in \eqref{eq:Laplace right-hand side}).
We now prove 
 the remaining cases of $\A$ and $\C$ (i.e., $\A+\C<0$, and either $\A=\C$ or $\A\C=0$). We choose four sequences $(\A'_j), (\A''_j),(\C'_j),(\C''_j)$ in $\R$ such that 
\begin{equation}
    \label{eq:avre}\lim_{j\to\infty}\A'_j=\lim_{j\to\infty}\A''_j=\A,\quad \lim_{j\to\infty}\C'_j=\lim_{j\to\infty}\C''_j=\C,
    \end{equation}
and that for $j=1,2,\dots$,
\begin{equation} 
    \label{eq:ass'}\A_j'>\A>\A_j'',\hspace{.15cm}\C_j'<\C<\C_j'',\hspace{.15cm} \A_j'+\C_j'<0,\hspace{.15cm} \A_j''+\C_j''<0,\hspace{.15cm} \A_j'\neq\C_j',\hspace{.15cm} \A_j''\neq\C_j'',\hspace{.15cm} \A_j',\A_j'',\C_j',\C_j''\neq0.\end{equation}

Suppose that the sequence $A_n,B_n,C_n,D_n$ satisfy Assumption \ref{assump:0}, 
that is,
\[
1-A_n\sim\frac{\A}{\sqrt{n}},\quad B_n\sim B,\quad 1-C_n\sim\frac{\C}{\sqrt{n}},\quad D_n\sim D,
\]
where $B,D\in(-1,0]$. For each $j\in\N$, we choose sequences $(A_{n,j}'), (A_{n,j}''),(C_{n,j}'),(C_{n,j}'')$ satisfying:
\[
1-A_{n,j}'\sim\frac{\A_j'}{\sqrt{n}},\quad1-A_{n,j}''\sim\frac{\A_j''}{\sqrt{n}},\quad1-C_{n,j}'\sim\frac{\C_j'}{\sqrt{n}},\quad 1-C_{n,j}''\sim\frac{\C_j''}{\sqrt{n}}\quad \text{ as }n\to\infty.
\]
Since $\A_j'>\A>\A_j''$ and $\C_j'<\C<\C_j''$, one can assume that for each $n,j\geq1$,
\[
A_{n,j}'\leq A_n\leq A_{n,j}'',\quad C_{n,j}'\geq C_n\geq C_{n,j}''.
\]
We also let:
\[
B_{n,j}'=B_{n,j}'':=B_n,\quad D_{n,j}'=D_{n,j}'':=D_n.
\]
By the stochastic sandwiching argument stated as Lemma \ref{lem:sandwiching}, for any $\vvc\in\R_+^d$ and $n,j\geq1$, we have:
\begin{equation}\label{eq:compare}
    \varphi_{n,\vvx}^{(A'_{n,j},B'_{n,j},C'_{n,j},D'_{n,j},q)}\pp{\frac\vvc{\sqrt n}}\geq
    \varphi_{n,\vvx}^{(A_n,B_n,C_n,D_n,q)}\pp{\frac\vvc{\sqrt n}}\geq
    \varphi_{n,\vvx}^{(A''_{n,j},B''_{n,j},C''_{n,j},D''_{n,j},q)}\pp{\frac\vvc{\sqrt n}}.
\end{equation}
On the other hand, observe that for each $j=1,2,\dots$, the sequence $A'_{n,j}$, $B'_{n,j}$, $C'_{n,j}$, $D'_{n,j}$ (resp. $A''_{n,j}$, $B''_{n,j}$, $C''_{n,j}$, $D''_{n,j}$) satisfy Assumption \ref{assump:0} for $\A'_j$, $\C'_j$ (resp. $\A''_j$, $\C''_j$). In view of our assumption \eqref{eq:ass'}, the convergence of Laplace transforms has been proved for $(\A'_j, \C'_j)$ and $(\A''_j, \C''_j)$,  
that is, for each $j=1,2,\dots$ we have:
\begin{align*}
\limn\varphi_{n,\vvx}^{(A'_{n,j},B'_{n,j},C'_{n,j},D'_{n,j},q)}\pp{\frac\vvc{\sqrt n}}&=\varphi_\vvx^{(\A_j',\C_j')}(\vvc),\\
\limn\varphi_{n,\vvx}^{(A''_{n,j},B''_{n,j},C''_{n,j},D''_{n,j},q)}\pp{\frac\vvc{\sqrt n}}&=\varphi_\vvx^{(\A_j'',\C_j'')}(\vvc), 
\end{align*} 
for any $\vvc=(c_1,\dots,c_d)\in\R_+^d$ satisfying  $-c_d<\inf\{a_j',a_j'':j=1,2,\dots\}\neq-\infty$.
Therefore, in view of \eqref{eq:compare}, 
\begin{align*}
\varphi_\vvx^{(\A_j'',\C_j'')}(\vvc)
&\leq
\liminf_{n\to\infty}\varphi_{n,\vvx}^{(A_n,B_n,C_n,D_n,q)}\pp{\frac\vvc{\sqrt n}}
\leq
\limsup_{n\to\infty}\varphi_{n,\vvx}^{(A_n,B_n,C_n,D_n,q)}\pp{\frac\vvc{\sqrt n}}\\
&
\leq
\varphi_\vvx^{(\A_j',\C_j')}(\vvc).
\end{align*}
Taking $j\to\infty$ in the above inequality, in view of \eqref{eq:avre} and Proposition \ref{prop:continuity of Laplace}, we have:
\[\limn\varphi_{n,\vvx}^{(A_{n},B_{n},C_{n},D_{n},q)}\pp{\frac\vvc{\sqrt n}}=\varphi_\vvx^{(\A,\C)}(\vvc)\]
for any $\vvc=(c_1,\dots,c_d)\in\R_+^d$ satisfying  $-c_d<\inf\{a_j',a_j'':j=1,2,\dots\}\neq-\infty$. By \cite[Theorem A.1]{bryc19limit} we conclude the proof of Theorem \ref{thm:1} for parameters $\A$ and $\C$ such that $\A+\C<0$.
\end{proof}

\appendix
  \section{Explicit formula for the normalization constant}\label{app:normalization}
  We prove \eqref{eq:calH}, which for convenience we re-stated in the following lemma.
\begin{lemma}
    For any $\A,\C\in\R$, we have:
    \begin{equation}\label{eq:vrerr}
    \esp_{\rm Bm}\lee e^{(\A+\C)\min_{t\in[0,1]}\omega_t - \A\omega_1}\ree=\calH(\A,\C)= \begin{cases}
\displaystyle \frac{\A H(\A/\sqrt 2)-\C H(\C/\sqrt 2)}{\A-\C}, & \mbox{ if } \A\ne \C,\\\\
\displaystyle (1+\A^2)H\pp{\frac {\A}{\sqrt 2}}-\sqrt{\frac2\pi}\A, & \mbox{ if } \A=\C,
\end{cases}
\end{equation}
where for $x\in\R$, $H(x) =e^{x^2}\erfc(x)$, where $\erfc (x) = \frac 2{\sqrt\pi}\int_x^\infty e^{-t^2}\d t$.
\end{lemma}
\begin{proof}
    By the reflection principle, we have:
\begin{equation*} 
    \begin{split}
        \proba\pp{\min_{s\in[0,1]}\mathbb B_s\in \d b, \mathbb B_1\in \d y\mmid \mathbb B_0 = 0}  
& = \proba\pp{\max_{s\in[0,1]}\mathbb B_s\in \d (-b), \mathbb B_1\in \d (-y)\mmid \mathbb B_0 = 0} \\
& = \sqrt{\frac{2}{\pi}}(y-2b)e^{-(y-2b)^2/2}  \inddd{b<0,b<y}.  
    \end{split}
\end{equation*} 
Formula \eqref{eq:vrerr} then follows from the following result. 
\end{proof}

\begin{lemma}\label{lem:berar}
For all $\alpha,\beta\in\R$, 
\begin{multline*} 
    \sqrt{\frac{2}{\pi}} \int_{-\infty}^0\int_b^\infty e^{(\alpha+\beta)b-\beta x}   (x-2b)e^{-(x-2b)^2/2} \d x\d b 
\\     =  
    \begin{cases}
\ds \frac1{\alpha-\beta}  \pp{\alpha H\pp{\frac\alpha{\sqrt 2}} - \beta H\pp{\frac\beta{\sqrt 2}}}, & \mbox{ if } \alpha\ne\beta, \\
\ds (1+\alpha^2)H\pp{\frac\alpha{\sqrt 2}} - \sqrt{\frac 2\pi}\alpha, & \mbox{ if } \alpha=\beta.
\end{cases}
\end{multline*} 
\end{lemma}
\begin{proof}
Recall first \citep[2.2.1.8]{prudnikov92integrals}, which states that for any $p\in\R$,
\begin{equation}\label{eq:2.2.1.8}
    p H\pp{\frac p{\sqrt 2}}
 =  p\sqrt{\frac 2\pi}\int_0^\infty e^{-x^2/2-px}\d x =\sqrt{\frac 2{\pi }}\int_0^\infty (1-e^{-px})xe^{-x^2/2}\d x.
\end{equation}  
Notice also that, for any $\alpha,\beta\in\R$,
\begin{equation}\label{eq:bvea}
    \begin{split}
        \int_0^\infty e^{-\beta x}\int_0^x e^{-(\alpha-\beta)r}  xe^{-x^2/2}\d r\d x  
& = \int_0^\infty e^{-(\alpha-\beta)r}\int_r^\infty e^{-\beta x}xe^{-x^2/2}\d x \d r \\
& = \int_{-\infty}^0 e^{(\alpha-\beta)b} \int_{-b}^\infty e^{-\beta x}xe^{-x^2/2}\d x\d b \\
& = \int_{-\infty}^0 e^{(\alpha+\beta)b} \int_{b}^\infty e^{-\beta x} (x-2b)e^{-(x-2b)^2/2}   \d x\d b.
    \end{split}
\end{equation} 
When $\alpha\neq\beta$, combining \eqref{eq:2.2.1.8} and \eqref{eq:bvea}, we have:
\begin{align*}
    \frac1{\alpha-\beta}  \pp{\alpha H\pp{\frac\alpha{\sqrt 2}} - \beta H\pp{\frac\beta{\sqrt 2}}}
& = \sqrt{\frac 2{\pi }}\int_0^\infty\frac{e^{-\beta x}-e^{-\alpha x}}{\alpha-\beta} xe^{-x^2/2}\d x \\
&= \sqrt{\frac 2{\pi }}\int_0^\infty e^{-\beta x}\int_0^x e^{-(\alpha-\beta)r}  xe^{-x^2/2}\d r\d x  \\
& = \sqrt{\frac{2}{\pi}} \int_{-\infty}^0\int_b^\infty e^{(\alpha+\beta)b-\beta x} (x-2b)e^{-(x-2b)^2/2}   \d x\d b.
\end{align*} 
When $\alpha= \beta$, by the identity \citep[(2.2.1.7)]{prudnikov92integrals}, we have:
\begin{multline*}
    (1+\alpha^2)H\pp{\frac\alpha{\sqrt 2}} - \sqrt{\frac 2\pi}\alpha=\sqrt{\frac 2\pi} \int_0^\infty e^{-\alpha x}x^2 e^{-x^2/2}\d x\\
    =\sqrt{\frac 2\pi}\int_0^\infty e^{-\alpha x}\int_0^x  xe^{-x^2/2}\d r\d x= \sqrt{\frac{2}{\pi}} \int_{-\infty}^0\int_b^\infty e^{2\alpha b-\beta x} (x-2b)e^{-(x-2b)^2/2}   \d x\d b.
\end{multline*}
where the last step uses \eqref{eq:bvea}. 
\end{proof}

\section{Duality formula of Laplace transforms of Markov processes}\label{sec:duality}

As mentioned in Remark \ref{rem:duality}, Proposition \ref{prop:duality} can be viewed as a duality formula for the Laplace transforms of Markov processes; see \citep{kuznetsov24dual} for a general framework. Unlike in several other cases involving limits of the open ASEP \citep{bryc19limit,wang24askey,bryc23asymmetric,bryc23markov} that we have encountered in the past, it is unexpected that we are unable to prove the duality formula by a direct calculation, except for the special case $d=1$.   
Here, we provide a direct and quick proof of the case $d=1$, and explain why even the case $d=2$ is a non-trivial task to prove directly, despite its simple formulation.

We first restate Proposition \ref{prop:duality} in the case $d=1$, and provide a proof by direct calculation. In the case $d=1$, we have $x_1=1$ and $s:=c_1>0$ with $-s<\A<-\C$.  Recall for convenience
\begin{align}
    \label{eq:psi11}
 \Phi_{1,1}(1,s) &  =  \frac{\ds 2\C-2s}{\ds \C-2s-\A} e^{(s-\C)^2/4},\\
    \label{eq:psi10}
\Phi_{1,0}(1,s) &  =  \frac{\ds \A+\C}{\ds \pi}\int_{\R_+} \frac{\ds \sqrt u}{\ds (\C-s)^2+u)((\A+s)^2+u)}e^{-u/4}\d u,
\end{align}
and that from Proposition \ref{prop:BLD} the random variable $\eta\topp{\A/\sqrt2,\C/\sqrt 2}_1$ has probability density function:
\begin{equation}
    \label{eq:d=1 density}
 p\topp{\A/\sqrt2,\C/\sqrt 2}_1(z)=\frac{ 1}{ \calH(\A/\sqrt2,\C/\sqrt 2)}\int_{-\infty}^0e^{(\A+\C)b/\sqrt{2}-\A z/\sqrt{2}}\q_1^{(b),*}(0,z)\d b,\end{equation}
where
\begin{equation}\label{eq:qd1}
\q_1^{(b),*}(0,z)=\sqrt{\frac 2{\pi }} (z-2b)e^{-(z-2b)^2/2}\inddd{b<0,b<z}.\end{equation}
We have the following.  
\begin{proposition}\label{prop:d=1}
Assume $\A+\C<0$ and $-s<\A<-\C$, we have:
\begin{equation}
    \label{eq:eebrvatbe}
\frac{\Phi_{1,0}(x,s)+\Phi_{1,1}(x,s)}{\calH\pp{\A/\sqrt{2},\C/\sqrt{2}}} = \frac{\calH\pp{(\A+s)/\sqrt{2}, (\C-s)/\sqrt{2}}}{\calH\pp{\A/\sqrt{2},\C/\sqrt{2}}} = \esp\lee e^{-s\eta\topp{\A/\sqrt2,\C/\sqrt 2}_1}\ree.
\end{equation}
\end{proposition}
\begin{proof}
Assume additionally $\A\neq\C$ and $\A,\C\neq0$. These constraints can be dropped by a continuity argument, and we omit the details.

Using the identity $H(x)+H(-x) = 2e^{x^2}$, by \eqref{eq:psi11} we have: 
\begin{align*}
\Phi_{1,1}(1,s)   =   \frac{2\C-2s}{\C-2s-\A}\pp{H\pp{\frac{s-\C}2} + H\pp{-\frac{s-\C}2}},
\end{align*}
Using the identity \citep[the equation after (4.38)]{bryc23asymmetric}:
\[
\int_0^\infty \frac{\ds \sqrt u}{\ds (a^2+u)(c^2+u)}e^{-u/4}\d u = \pi\frac{\ds cH(c/2)-aH(a/2)}{\ds c^2-a^2},\quad  \mfa a,c>0,
\]
by \eqref{eq:psi10} we have: 
\begin{equation*}
\Phi_{1,0}(1,s)
 =  \frac{\ds (s-\C)H((s-\C)/2)-(\A+s)H((\A+s)/2)}{\ds \C-\A-2s},
\end{equation*}
Combining the above identities, the first equality of \eqref{eq:eebrvatbe} follows. Set $\alpha = (\C-s)/\sqrt 2$ and $\beta = (\A+s)/\sqrt 2$ in Lemma \ref{lem:berar}, in view of \eqref{eq:d=1 density} and \eqref{eq:qd1}, the second equality follows. We conclude the proof.
\end{proof}

We next  provide a variation of Proposition \ref{prop:duality}. Recall expressions of $\Phi_{d,\ell}$ in \eqref{eq:wt dd}, \eqref{eq:wt d0}, and \eqref{eq:wt dl}, which are integrals involving transition kernel $\p$. We also introduce
\begin{multline*}
\Psi_{d,\ell}(\vvx,\vvc) 
:= 
\int_{-\infty}^0 e^{(\A+\C)b/\sqrt 2}\int_{(b,\infty)^d} e^{-\summ k1dc_kz_k/\sqrt 2-\A z_2/\sqrt 2}\\
\times \q_{\Delta x_\ell}^{(b),*}(z_{\ell-1},z_\ell)\prod_{\substack{k=1,\dots,d\\k\ne\ell}}\q_{\Delta x_k}\topp b(z_{k-1},z_k)\d \vvz_{1:d}\d b,
\end{multline*}
with $\ell = 1,\dots,d$. These are integrals involving transition kernels $\q$ and $\q^{(b),*}$. Note that by \eqref{eq:BLD} and \eqref{eq:Shepp1}, we have
\[ 
\frac{\Psi_{d,\ell}(\vvx,\vvc)}{\calH(\A/\sqrt 2,\C/\sqrt 2)} = \esp\pp{\exp\pp{-\frac1{\sqrt{2}}\summ k1d c_k\eta\topp{\A/\sqrt 2,\C/\sqrt 2}_{t_k}}\inddd{\argmin_{t\in[0,1]}\eta\topp{\A/\sqrt 2,\C/\sqrt 2}_t\in (t_{\ell-1},t_\ell]}}.
\] 
Now, we can rewrite Proposition \ref{prop:duality} as the following.
\begin{proposition}\label{prop:duality1}With $c_1,\dots,c_d>0, -c_d<\A<-\C$, $0=x_0<x_1<\cdots<x_d=1$, 
\equh\label{eq:conjecture1}
\summ \ell0d\Phi_{d,\ell}(\vvx,\vvc) = \summ \ell1d\Psi_{d,\ell}(\vvx,\vvc).
\eque
In the above, we have $\Phi_{d,\ell}$ given as in \eqref{eq:wt dd}, \eqref{eq:wt d0}, and \eqref{eq:wt dl}, and the last two expressions can be replaced by  \eqref{eq:Phi d0} and \eqref{eq:Phi dl} respectively below.
\end{proposition}
\begin{remark}It would be nice to have a direct proof of \eqref{eq:conjecture1}, which can be viewed as an identity for the Laplace transform of the process $\eta\topp{\A,\C}$. We were hoping that it might be easier to prove the identity using \eqref{eq:Phi d0} and \eqref{eq:Phi dl} because of the following observation. The right-hand side are integrals involving transitional kernels $\q$ and $\q^{(b),*}$, the expressions of $\Phi_{d,\ell}$ we first obtained are in terms of kernels $\p$, and the expressions in \eqref{eq:Phi d0} and \eqref{eq:Phi dl} are in terms of $\q$. Hence in this view new expressions are {\em closer} to those expressions of $\Psi_{d,\ell}$.
However, we are still unable to prove \eqref{eq:conjecture1} directly, even with $d=2$. 
\end{remark}

We explain how to rewriting \eqref{eq:wt d0} and \eqref{eq:wt dl} by \eqref{eq:Phi d0} and \eqref{eq:Phi dl} respectively.
 For this purpose we will use a version of the general duality formula, namely \citep[Proposition 4.4]{bryc23asymmetric},  that was helpful in the fan region. The statement of this formula is notationally long and hence omitted.
 
 For $\Phi_{d,0}$ by duality formula and \eqref{eq:wt d0} we have
\equh\label{eq:Phi d0}
\Phi_{d,0}(\vvx,\vvc) = \frac{\A+\C}{\sqrt 2}\int_{\R^{d+1}}e^{-\summ k1d c_kz_k/\sqrt 2 + (\C-s_1)z_0/\sqrt 2 - \A z_d/\sqrt 2}\prodd k1d \q_{\Delta x_k}(z_{k-1},z_k)\d\vvz_{0:d}.
\eque
For $\Phi_{d,\ell}$, we first notice that in \eqref{eq:wt dl}, we can re-write 
\begin{multline*}
\frac{4(\C-s_\ell)(s_{\ell+1}-s_\ell)}{((\C-s_{\ell+1})^2+u_{\ell+1})((\C-2s_\ell+s_{\ell+1})^2+u_{\ell+1})} \\
= \frac1{(\C-s_{\ell+1})^2+u_{\ell+1}} - \frac1{(\C-2s_\ell+s_{\ell+1})^2+u_{\ell+1}}. 
\end{multline*}
We then write \eqref{eq:wt dl} as the difference of two $(d-\ell)$-multiple integrals, and apply the duality formula respectively. We end up with
\equh\label{eq:Phi dl}
\Phi_{d,\ell}(\vvx,\vvc) = \Phi_{d,\ell,+}(\vvx,\vvc)- \Phi_{d,\ell,-}(\vvx,\vvc),
\eque
with
\begin{align*}
\Phi_{d,\ell,+}(\vvx,\vvc)& = \frac{\A+\C}{\sqrt 2}\exp\pp{\frac14\summ k1\ell\Delta {x_k}(s_k-\C)^2}\\
& \quad \times \int_{\R_+^{d-\ell+1}}e^{-\sum_{k=\ell+1}^d c_kz_k/\sqrt 2 + (\C-s_{\ell+1})z_\ell/\sqrt 2 - \A z_d/\sqrt 2}\prod_{k=\ell+1}^d\q_{\Delta x_k}(z_{k-1},z_k)\d \vv z_{\ell:d},
\end{align*}
and\begin{align*}
\Phi_{d,\ell,-}(\vvx,\vvc)& = \frac{\A+\C}{\sqrt 2}\exp\pp{\frac14\summ k1\ell\Delta {x_k}(s_k-\C)^2}\\
& \quad \times\int_{\R_+^{d-\ell+1}}e^{- 
\sum\limits_{k=\ell+1}^d
 c_kz_k/\sqrt 2 + (\C+s_{\ell+1}-2s_\ell)z_\ell/\sqrt 2 - \A z_d/\sqrt 2}\prod_{k=\ell+1}^d\q_{\Delta x_k}(z_{k-1},z_k)\d \vv z_{\ell:d},\\
\end{align*}

For convenience, we also provide the identity we are unable to prove {\em by a direct calculation} when $d=2$ (see Proposition \ref{prop:duality} for the constraints on $\A,\C,\vvx,\vvc$). The identity of which we would like to have a direct proof is the following
\[
\Phi_{2,0}+\Phi_{2,1,+}-\Phi_{2,1,-}+\Phi_{2,2} = \Psi_{2,1}+\Psi_{2,2},
\]
with
\begin{align*}
\Phi_{2,0}(\vvx,\vvc) &= \frac{\A+\C}{\sqrt 2}\int_{\R_+^{3}}e^{-\summ k12 c_kz_k/\sqrt 2 + (\C-s_1)z_0/\sqrt 2 - \A z_2/\sqrt 2}\prodd k12 \q_{\Delta x_k}(z_{k-1},z_k)\d\vvz_{0:2},\\
\Phi_{2,1,+}(\vvx,\vvc)& = \frac{\A+\C}{\sqrt 2}e^{\Delta {x_1}(s_1-\C)^2/4} \int_{\R_+^{2}}e^{-c_2z_2/\sqrt 2 + (\C-s_2)z_1/\sqrt 2 - \A z_2/\sqrt 2}\q_{\Delta x_2}(z_{1},z_2)\d \vv z_{1:2},\\
\Phi_{2,1,-}(\vvx,\vvc)& = \frac{\A+\C}{\sqrt 2}e^{\Delta {x_1}(s_1-\C)^2/4}\int_{\R_+^{2}}e^{- c_2z_2/\sqrt 2 + (\C+s_2-2s_1)z_1/\sqrt 2 - \A z_2/\sqrt 2}\q_{\Delta x_2}(z_{1},z_2)\d \vv z_{1:2},\\
\Phi_{2,2}(\vvx,\vvc) &=   \frac{2\C-2s_2}{\C-2s_2-\A}e^{\summ k12 \Delta x_k(s_k-\C)^2/4},
\end{align*}
and
\begin{align*}
\Psi_{2,1}(\vvx,\vvc) &=\int_{-\infty}^0 e^{(\A+\C)b/\sqrt 2}\int_{(b,\infty)^2} e^{-c_1z_1/\sqrt 2}\q_{\Delta {x_1}}^{(b),*}(0,z_1)e^{-(\A+c_2)z_2/\sqrt 2}\q_{\Delta x_2}\topp b(z_1,z_2)\d z_1\d z_2\d b,\\
\Psi_{2,2}(\vvx,\vvc)& =\int_{-\infty}^0 e^{(\A+\C)b/\sqrt 2}\int_{(b,\infty)^2} e^{-c_1z_1/\sqrt 2}\q_{\Delta {x_1}}^{(b)}(0,z_1) e^{-(\A+c_2)z_2/\sqrt 2}\q_{\Delta x_2}^{(b),*}(z_1,z_2)\d z_1\d z_2\d b.
\end{align*}

\bibliographystyle{apalike}   
\def\cprime{$'$} \def\polhk#1{\setbox0=\hbox{#1}{\ooalign{\hidewidth \lower1.5ex\hbox{`}\hidewidth\crcr\unhbox0}}} \def\polhk#1{\setbox0=\hbox{#1}{\ooalign{\hidewidth \lower1.5ex\hbox{`}\hidewidth\crcr\unhbox0}}}

\begin{acks}
We would like to thank W\l odek Bryc and Alexey Kuznetsov for several helpful and stimulating
 discussions. 
\end{acks}

\end{document}